\documentclass[a4paper,11pt]{amsart}
\usepackage{amsfonts,amssymb,amscd,amsmath,latexsym,amsbsy,enumerate,
tikz,  a4wide, ifthen, mathrsfs}

\usepackage{todonotes}
\usetikzlibrary{calc}
\newcommand*\circled[1]{\tikz[baseline=(char.base)]{
            \node[shape=circle,draw,inner sep=2pt] (char) {#1};}}

\theoremstyle{plain}
\newtheorem{thm}{Theorem}[section]

\newtheorem{lem}[thm]{Lemma}
\newtheorem{cor}[thm]{Corollary}
\newtheorem{prop}[thm]{Proposition}

\theoremstyle{definition}
\newtheorem{rmk}[thm]{Remark}

\numberwithin{equation}{section}

\def\al{\alpha}
\def\be{\beta}

\def\de{\delta}
\def\la{\lambda}

\def\R{\mathbb{R}}
\def\C{\mathbb{C}}
\def\N{\mathbb{N}}

\def\Uq{\mathcal{U}_q(\mathfrak{su}(3))}

\def\pfq#1#2#3#4#5#6{%
	{}_{#1}\phi_{#2}\biggl(\genfrac..{0pt}{}{#3}{#4}\biggr|#5,#6\biggr)%
}

\newcommand{\SU}{\mathrm{SU}}

\newcommand{\diag}{\mathrm{diag}}

\def\tensor{\otimes}
\newcommand{\tr}{\mathrm{tr}}

\title[Branching rules]
{Branching rules for finite-dimensional $\Uq$-representations
with respect to a right coideal subalgebra}

\author{Noud Aldenhoven}
\author{Erik Koelink}
\address{IMAPP, Radboud Universiteit,
PO Box 9010, 6500 GL Nijmegen,
the Netherlands}
\email{n.aldenhoven@math.ru.nl, e.koelink@math.ru.nl}
\author{Pablo Rom\'an}
\address{CIEM,
FaMAF, Universidad Nacional de C\'ordoba, Medina Allende s/n Ciudad
Universitaria, C\'ordoba, Argentina}
\email{roman@famaf.unc.edu.ar}
\date{\today}

\begin{document}

\begin{abstract}
We consider the quantum symmetric pair $(\Uq, \mathcal{B})$ where $\mathcal{B}$ is
a right coideal subalgebra. We prove that all finite-dimensional irreducible
representations of $\mathcal{B}$ are weight representations and are characterised
by their highest weight and dimension.

We show that the restriction of a finite-dimensional irreducible
representation of $\Uq$ to $\mathcal{B}$ decomposes multiplicity free into
irreducible representations of $\mathcal{B}$. Furthermore we give explicit expressions
for the highest weight vectors in this decomposition in terms of dual $q$-Krawtchouk polynomials.
\end{abstract}

\maketitle

\section{Introduction}\label{sec:intro}

The theory of quantum symmetric pairs of Lie groups has been developed by Koornwinder, Dijkhuizen, Noumi and Sugitani and others \cite{DN98, DS99, Nou96, NDS97, NS95, Sug99} for classical Lie groups and by G. Letzter \cite{Let99, Let02, Let03, Let04,  Let08} for all semisimple Lie algebras, see also \cite{Kol14}.
The motivating example for the development for this theory was given by Koornwinder \cite{Koo93}, who studied scalar-valued spherical functions on the quantum analogue of $(\SU(2), \textrm{U}(1))$ considering twisted primitive elements in the quantised universal enveloping algebra of $\mathcal{U}_q(\mathfrak{sl}(2))$.
Koornwinder identified all scalar-valued spherical functions with Askey-Wilson polynomials in two free parameters. Dijkhuizen and Noumi \cite{DN98} extended the work of Koornwinder  to quantum analogues of $(\SU(n+1), \textrm{U}(n))$ considering two sided coideals of the quantised universal enveloping algebra of $\mathcal{U}_q(\mathfrak{gl}({n+1}))$. More generally, Letzter  considered the quantised universal enveloping algrebra $\mathcal{U}_q(\mathfrak{g})$ with a right coideal subalgebra $\mathcal{B}$, which is the quantum analogue of $\mathcal{U}(\mathfrak{k})$ for a Cartan decomposition $\mathfrak{g}=\mathfrak{k}\oplus\mathfrak{p}$. In \cite{Let04} all scalar-valued spherical functions for quantum symmetric pairs with reduced restricted root systems are identified with Macdonald polynomials. However, the requirement of having a reduced restricted root system excludes the quantum analogue of $(\SU(3), \mathrm{U}(2))$.

One recent extension of this situation \cite{AKR15}  arises with the study of matrix-valued spherical functions of the quantum analogue of $(\SU(2) \times \SU(2), \diag)$ where higher-dimensional representations of coideal subalgebra $\mathcal{B}$ are involved. The quantum symmetric pair is given by the quantised universal enveloping algebra of  $\mathcal{U}_q(\mathfrak{g})$, where $\mathfrak{g} = \mathfrak{su}(2) \oplus \mathfrak{su}(2)$, and a right coideal subalgebra $\mathcal{B}$ than can be identified with $\mathcal{U}_q(\mathfrak{su}(2))$. As in the Lie group setting \cite{KvPR12, KvPR13,vPrui12,HvP13}, the explicit knowledge of the branching rules plays a fundamental role in the explicit determination of the matrix-valued spherical functions. In this first case, the branching rules for the irreducible representations of $\mathcal{U}_q(\mathfrak{g})$ with respect to $\mathcal{B}$ follow using the standard Clebsch-Gordan decomposition.

One of the first technical difficulties that one finds to extend the results of \cite{AKR15} to more general quantum symmetric pairs
is the lack of the explicit branching rules for finite-dimensional $\mathcal{U}_q(\mathfrak{g})$-representations with respect to a right coideal subalgebra.
In this paper we deal with this problem for the quantised universal enveloping algebra $\Uq$ with a right coideal subalgebra
$\mathcal{B}$ as in Kolb \cite{Kol14}.  We study the problem of describing all irreducible representations that occur
in the restriction to $\mathcal{B}$ of finite-dimensional irreducible representations of $\Uq$. In general, information about branching rules for quantum symmetric pairs $(\mathcal{U}_q(\mathfrak{g}),\mathcal{B})$ as in Kolb \cite{Kol14} and Letzter \cite{Let99,Let02} is relatively scarce in particular in case the coideal subalgebra depends on an additional parameter as in this paper. However see Oblomkov and Stokman \cite{OS} for partial information on the branching rules for the quantum analogue of $(\mathfrak{gl}(2n),\mathfrak{gl}(n)\oplus \mathfrak{gl}(n))$.

This paper is organised as follows.
In Section \ref{sec:uqsl3} we review the construction of the quantised universal enveloping algebra $\Uq$ and its finite
dimensional irreducible representations. Then we collect a series of commutation identities for the generators of $\Uq$ and we introduce a an orthogonal basis for finite-dimensional $\Uq$-representations which is an analogue of Mudrov \cite{Mudr}. We also describe the action of the generators of $\Uq$ on this basis. In Section \ref{sec:coideal-subalgebra} we fix a right coideal subalgebra $\mathcal{B}$ of the quantised universal enveloping algebra which depends on two complex parameters $c_1, c_2$. We describe the generators of the Cartan subalgebra of $\mathcal{B}$ and we use them to classify all finite-dimensional irreducible representations of $\mathcal{B}$ under a mild genericity condition on the parameters. More precisely we prove that every finite-dimensional irreducible representation of $\mathcal{B}$ is completely characterised by its highest weight and its dimension. In Section \ref{sec:branching-rule} we prove the main theorem of the paper.  We show that any irreducible finite-dimensional representation of $\Uq$ decomposes multiplicity free into irreducible representations of the $\mathcal{B}$ and we characterise the representations that occur in the decomposition by their highest weight and dimension. The highest weight vectors of the coideal subalgebra $\mathcal{B}$-representations are obtained by diagonalising an element of the Cartan subalgebra of $\mathcal{B}$ restricted to a certain subspace where it acts tridiagonally. The eigenvectors can be then identified explicitly in terms of dual $q$-Krawtchouk polynomials.

\section{The quantised universal enveloping algebra $\Uq$}
\label{sec:uqsl3}
Let $\mathfrak{g}=\mathfrak{sl}(3)=\{X\in\mathfrak{gl}(3,\C): \, \tr(X)=0\}$. We fix the Cartan subalgebra $\mathfrak{h}$ of diagonal matrices. Let $A = (a_{i,j})_{i,j}$ be the Cartan matrix for $\mathfrak{g}$, i.e. $a_{i,i}=2$, $i=1,2$, and $a_{i,j}$=-1 for $i\neq j$.
Let $R\subset\mathfrak{h}$ denote the root system of $\mathfrak{g}$. We denote by $R^+$ the subset of positive roots, so that we have the decomposition $\mathfrak{g}=\mathfrak{n}^-\oplus \mathfrak{h}\oplus\mathfrak{n}^+$. We denote by $(\cdot,\cdot)$ the canonical inner product on $\mathfrak{h}$ and by $\Pi=\{\alpha_1,\alpha_2\}$ the simple roots so that $(\alpha_i, \alpha_j) = a_{i,j}$.
The fundamental weights are given by $\varpi_1 = \frac{2}{3}\alpha_1 + \frac{1}{3}\alpha_2$ and $\varpi_2 = \frac{1}{3}\alpha_1 + \frac{2}{3}\alpha_2$.

The quantised universal enveloping algebra $\Uq$ is the unital associative algebra generated by $E_i$, $F_i$ and $K_i^{\pm 1}$, where $i = 1, 2$, subject to the relations
\begin{equation}
\label{eq:uqsl3-relations}
\begin{split}
K_i^{\pm 1} K_j^{\pm 1} &= K_{j}^{\pm 1} K_{i}^{\pm 1}, \quad K_i^{\pm 1} K_j^{\mp 1} = K_{j}^{\mp 1} K_{i}^{\pm 1}, \quad K_i K_i^{-1} = 1 = K_i^{-1} K_i, \\
K_i E_j &= q^{(\alpha_i, \alpha_j)} E_j K_i, \quad K_i F_j = q^{-(\alpha_i, \alpha_j)} F_j K_i, \quad \left[E_i, F_j\right] = \frac{K_i - K_i^{-1}}{q - q^{-1}} \delta_{i,j},
\end{split}
\end{equation}
for $i, j = 1, 2$ and, for $i \neq j$, the quantum Serre's relations
\begin{equation}
\label{eq:qSerre}
E_i^2 E_j - (q + q^{-1}) E_i E_j E_i + E_j E_i^2 = 0 = F_i^2 F_j - (q + q^{-1}) F_i F_j F_i + F_j F_i^2.
\end{equation}
We assume that $q\in[0,1]$. The quantised universal enveloping algebra $\Uq$ has a Hopf algebra structure with comultiplication $\Delta$, counit $\epsilon$ and antipode $S$ defined by
\begin{equation*}
\begin{split}
\Delta &: E_i, F_i, K_i^{\pm 1} \mapsto E_i \tensor 1 + K_i \tensor E_i, F_i \tensor K_i^{-1} + 1 \tensor F_i, K_i^{\pm 1} \tensor K_i^{\pm 1}, \\
\epsilon &: E_i, F_i, K_i^{\pm 1} \mapsto 0, 0, 1, \quad S : E_i, F_i, K_i^{\pm 1} \mapsto -K_i^{-1} E_i, -F_i K_i, K_i^{\mp 1},
\end{split}
\end{equation*}
with $i = 1, 2$. The $\ast$-structure on $\Uq$ is given by
\begin{equation}\label{eq:uqsl3-star}
E_i^\ast = K_i F_i, \quad F_i^\ast = E_i K_i^{-1}, \quad (K_i^{\pm 1})^\ast = K_i^{\pm 1}, \quad i = 1, 2,
\end{equation}
so that $\Uq$ is a Hopf $*$-algebra. Following Mudrov \cite{Mudr} we define for $a\in \R$
\begin{equation*}
\begin{split}
& F_3 = [F_1,F_2]_q = F_1F_2-qF_2F_1, \quad E_3 = [E_2,E_1]_q = E_2E_1-qE_1E_2, \\
& \hat{F}_3[a] = F_1F_2 \left(\frac{q^{a+1}K_2-q^{-a-1}K_2^{-1}}{q-q^{-1}} \right)
- F_2F_1 \left(\frac{q^{a}K_2-q^{-a}K_2^{-1}}{q-q^{-1}} \right), \\
& \hat{E}_3[a] = \left(\frac{q^{a+1}K_2-q^{-a-1}K_2^{-1}}{q-q^{-1}} \right) E_2E_1
-  \left(\frac{q^{a}K_2-q^{-a}K_2^{-1}}{q-q^{-1}} \right)E_1E_2, \\
\end{split}
\end{equation*}
and $\hat{F}_3=\hat{F}_3[0]$, $\hat{E}_3=\hat{E}_3[0]$.

\begin{lem}\label{lem:relinUq}
The following relations hold in $\Uq$:
\begin{enumerate}[(i)]
\setlength\itemsep{0.3em}
\item\label{lem:relinUq-i} $F_1\hat{F}_3[a] = \hat{F}_3[a]F_1,$
\item\label{lem:relinUq-iv} $E_2\hat{F}_3[a] = \hat{F}_3[a-2]E_2 - \frac{(q^{a}-q^{-a})}{(q-q^{-1})}F_1,$
\item\label{lem:relinUq-v} $K_i \hat{F}_3[a] = q^{-1} \hat{F}_3[a] K_i$, $K_i \hat{E}_3[a] = q \hat{E}_3[a] K_i$, $i=1,2.$
\end{enumerate}
\end{lem}

\begin{proof}
Straightforward verifications using \eqref{eq:uqsl3-relations} and  \eqref{eq:uqsl3-star}.
\end{proof}

\begin{lem}
\label{lem:EkFkmodUqEk} For $i=1,2$:
\begin{enumerate}[(i)]
\item \label{lem:EkFkmodUqEk-iii}
\[
E_i F_i^k =  F_i^k E_i + \frac{q^k-q^{-k}}{q-q^{-1}} F_i^{k-1}
\frac{q^{1-k}K_i-q^{k-1}K_i^{-1}}{q-q^{-1}}
\]
\item \label{lem:EkFkmodUqEk-ii}
\begin{equation*}
\begin{split}
E_i^{k}F_i^k &=
\frac{q^k\, (q^2;q^2)_k}{(1-q^2)^{2k}} (q^{2-2k} K_i^2;q^2)_k K_i^{-k} + \Uq \, E_i  \\
& =
\frac{(q^2;q^2)_k}{(1-q^2)^{2k}} (-1)^k q^{-k(k-2)} (K_i^{-2};q^2)_k K_i^{k} + \Uq \, E_i.
\end{split}
\end{equation*}
\end{enumerate}
\end{lem}

\begin{proof}
Straightforward verifications using \eqref{eq:uqsl3-relations} and  \eqref{eq:uqsl3-star}
and induction.
\end{proof}


\subsection{The finite-dimensional representations of $\Uq$}\label{sec:reprs-uqsl3}
Finite-dimensional representations of $\Uq$ are weight representations and are uniquely determined, up to equivalence, by their highest weights. Let $(\pi_\lambda,V_\lambda)$ be an irreducible finite-dimensional representation with highest weight $\lambda=\lambda_1 \varpi_1 + \lambda_2 \varpi_2$, $\la_1,\la_2\in \N$, and $v_\lambda$ a highest weight vector such that
\begin{equation}\label{eq:defrelhwVla}
E_i\, v_\la = 0, \quad K_i\, v_\la = q^{(\la, \al_i)} v_\la = q^{\lambda_i} v_\la.
\end{equation}
Then the dimension of $V_\lambda$ is the same as the dimension of the corresponding irreducible representation $\pi_\lambda$ of $\mathfrak{su}(3)$, namely
\begin{equation*}
\dim(V_\la) = \frac12 (\la_1+1)(\la_2+1)(\la_1+\la_2+2).
\end{equation*}
Furthermore for a weight $\nu=\nu_1 \,\varpi_1 + \nu_2 \, \varpi_2$, the dimension of the weight space
$$V_\lambda(\nu)=\{ v\in V_\lambda:\,\, K_i\, v = q^{(\nu,\alpha_i)} v,\,\, i=1,2\},$$
and the dimension of the weight space corresponding to the weight $\nu$  in the representation of $\mathfrak{su}(3)$ coincide, see \cite[Ch. 7]{KS97}. In particular, $\dim(V_\lambda(\lambda))=1$. The vector  space $V_\lambda$ is generated by the vectors $v_\lambda$ and $F_{i_1}F_{i_2}\ldots F_{i_m} \, v_\lambda$, $i_j\in \{1,2\}$ and is equipped with an inner product $\langle \cdot, \cdot\rangle$ that satisfies
\begin{equation*}
\langle v_\lambda,v_\lambda\rangle=1,\qquad \langle  X\, v, w\rangle = \langle   v, X^\ast \, w\rangle, \qquad
\forall\, X\in \Uq, \quad \forall\, v,w\in V_\la.
\end{equation*}

Mudrov  \cite{Mudr} describes the Shapovalov basis for the Verma modules of $\Uq$, and we have adapted his proof and construction to an orthonormal basis for the finite-dimensional unitary representations of $\Uq$. For completeness, we have sketched the proof in Appendix \ref{apen:proof_basis}. It is essentially due to Mudrov \cite[\S 8]{Mudr}.
\begin{thm}\label{thm:basisVla}
The set of vectors
$$\mathscr{B} = \{ F_2^k \hat{F}_3^l F_1^m\, v_\la \mid
0\leq m \leq \la_1, \ 0\leq l \leq \la_2, \ 0\leq k \leq \la_2+m-l \}$$
forms an orthogonal basis for $V_\la$. Explicitly,
\begin{equation*}
\langle F_2^k \hat{F}_3^l F_1^m\, v_\la, F_2^{k'} \hat{F}_3^{l'} F_1^{m'}\, v_\la \rangle =
\de_{k,k'} \de_{l,l'} \de_{m,m'}
H_{k,l,m},
\end{equation*}
where
\begin{multline*}
H_{k,l,m} = (q^2,q^{-2(\la_2-l+m)};q^2)_k
(q^2, q^{-2\la_1};q^2)_m
(q^2, q^{-2\la_2}, q^{-2(\la_2+1+m)}, q^{-2(\la_1+\la_2+1)};q^2)_l \\
\times (1-q^{2})^{-2(k+2l+m)}
(-1)^{k+l+m} q^{3(k+3l+m)} q^{-l(l-2m)}q^{-2l\lambda_2}.
\end{multline*}
\end{thm}
In Theorem \ref{thm:basisVla} we use the standard notation in \cite{GR} for $q$-shifted factorials
\begin{align*}
(q^a;q)_n&=(1-q^a)(1-q^{a+1})\ldots(1-q^{a+n-1}), \\
(q^{a_1},q^{a_2},\ldots ,q^{a_j};q)_n&=(q^{a_1},q)_n(q^{a_2},q)_n\ldots(q^{a_j},q)_n.
\end{align*}
Note that $H_{k,l,m}$ is indeed positive. In the following proposition we calculate the action of the generators of $\Uq$ in the basis $\mathscr{B}$ of Theorem \ref{thm:basisVla}.

\begin{prop}\label{prop:actiongeneratorsinbasis}
In the basis $\mathscr{B}$ of $V_\la$ as in Theorem \ref{thm:basisVla} we have
\begin{enumerate}[(i)]
\setlength\itemsep{0.3em}
\item \label{prop:gen_1} $K_1  F_2^k \hat{F}_3^l F_1^m\, v_\la =
q^{\la_1+k-l-2m} \, F_2^k \hat{F}_3^l F_1^m\, v_\la$,
\item  \label{prop:gen_2} $K_2  F_2^k \hat{F}_3^l F_1^m\, v_\la =  q^{\la_2-2k-l+m} \,F_2^k \hat{F}_3^l F_1^m\, v_\la$,
\item \label{prop:gen_3}  $F_1  F_2^k \hat{F}_3^l F_1^m\, v_\la = a_k(l,m) \, F_2^k \hat{F}_3^l F_1^{m+1}\, v_\la +
b_k(l,m)\, F_2^{k-1} \hat{F}_3^{l+1} F_1^{m}\, v_\la$,
\item \label{prop:gen_4} $E_1  F_2^k \hat{F}_3^l F_1^m\, v_\la = \al_k(l,m) \, F_2^k \hat{F}_3^l F_1^{m-1}\, v_\la +
\be_k(l,m) \, F_2^{k+1} \hat{F}_3^{l-1} F_1^{m}\, v_\la$,
\item \label{prop:gen_5} $F_2  F_2^k \hat{F}_3^l F_1^m\, v_\la = F_2^{k+1} \hat{F}_3^l F_1^m\, v_\la,$
\item \label{prop:gen_6} $E_2  F_2^{k} \hat{F}_3^l F_1^m\, v_\la = \eta_{k}(l,m) \,  F_2^{k-1} \hat{F}_3^l F_1^m\, v_\la$,
\end{enumerate}
with coefficients
\begin{gather*}
a_k(l,m) = \frac{(q^{\lambda_2+m+1-k-l}-q^{-\lambda_2-m-1+k+l})}{(q^{\lambda_2+m+1-l}-q^{-\lambda_2-m-1+l})},\qquad b_k(l,m) = \frac{(q^k-q^{-k})}{(q^{\lambda_2+m+1-l}-q^{-\lambda_2-m-1+l})}, \\
\eta_{k}(l,m) = \frac{q^k-q^{-k}}{q-q^{-1}}
\frac{q^{1-k+\la_2-l+m}-q^{k-1-\la_2+l-m}}{q-q^{-1}},\\
\al_k(l,m) = \frac{(q^m-q^{-m})(q^{\lambda_1-m+1}-q^{-\lambda_1+m-1})(q^{\lambda_2+m+1}-q^{-\lambda_2-m-1})}{(q-q^{-1})^2(q^{\lambda_2+m-l+1}-q^{-\lambda_2-m+l-1})},\\
\be_k(l,m) =\frac{(q^l-q^{-l})(q^{\lambda_2-l+1}-q^{-\lambda_2+l-1})(q^{\lambda_1+\lambda_2-l+2}-q^{-\lambda_1-\lambda_2+l-2})}{(q-q^{-1})^2(q^{\lambda_2+m-l+1}-q^{-\lambda_2-m+l-1})}.
\end{gather*}
\end{prop}

\begin{rmk} Note that the denominators in $a_k(l,m)$, $b_k(l,m)$, $\eta_k(l,m)$, $\alpha_k(l,m)$ and $\beta_k(l,m)$ are non-zero by the
ranges of $k,l,m$ as in Theorem \ref{thm:basisVla}.
\end{rmk}

\begin{proof}
The action of $K_i$, $i=1,2$, follows from  \eqref{eq:defrelhwVla}, \eqref{eq:uqsl3-relations} and Lemma \ref{lem:relinUq}\eqref{lem:relinUq-v}. The action of $F_2$ is trivial. The action of $E_2$ follows from
Lemma \ref{lem:EkFkmodUqEk}\eqref{lem:EkFkmodUqEk-iii},
Lemma \ref{lem:relinUq}\eqref{lem:relinUq-iv} and \eqref{eq:defrelhwVla}
and the established actions of $K_2$. This completes the proof of \eqref{prop:gen_1}, \eqref{prop:gen_2}, \eqref{prop:gen_5} and \eqref{prop:gen_6}.

In order to establish the action of $F_1$, we first show that there exist constants $a_k$ and $b_k$ so that
\begin{equation*}
F_1  F_2^k \hat{F}_3^l F_1^m\, v_\la =
a_k F_2^k \hat{F}_3^l F_1^{m+1}\, v_\la +
b_k F_2^{k-1} \hat{F}_3^{l+1} F_1^{m}\, v_\la
\end{equation*}
by induction with respect to $k$. The case $k=0$ with
$a_0=1$, $b_0=0$ is immediate from
Lemma \ref{lem:relinUq}\eqref{lem:relinUq-i}. In case $k=1$, we write
\begin{align*}
F_1  F_2 \hat{F}_3^l F_1^m\, v_\la &=
F_1 F_2 \frac{qK_2-q^{-1}K_2^{-1}}{q-q^{-1}}
\frac{q-q^{-1}}{q^{1+\la_2-l+m}-q^{-1-\la_2+l-m}}
\hat{F}_3^l F_1^m\, v_\la \\
&= \frac{q-q^{-1}}{q^{1+\la_2-l+m}-q^{-1-\la_2+l-m}} \left(
\hat{F}_3 + F_2F_1 \frac{K_2-K_2^{-1}}{q-q^{-1}} \right)
\hat{F}_3^l F_1^m\, v_\la \\
&=\frac{q^{\la_2-l+m}-q^{-\la_2+l-m}}{q^{1+\la_2-l+m}-q^{-1-\la_2+l-m}}
F_2 \hat{F}_3^l F_1^{m+1}\, v_\la \\
& \hspace{5cm}+
\frac{q-q^{-1}}{q^{1+\la_2-l+m}-q^{-1-\la_2+l-m}} \hat{F}_3^{l+1} F_1^m\, v_\la
\end{align*}
again using Lemma \ref{lem:relinUq}\eqref{lem:relinUq-i}. So the case $k=1$
is proved with
\[
a_1 = \frac{q^{\la_2-l+m}-q^{-\la_2+l-m}}{q^{1+\la_2-l+m}-q^{-1-\la_2+l-m}}, \quad
b_1 = \frac{q-q^{-1}}{q^{1+\la_2-l+m}-q^{-1-\la_2+l-m}}.
\]
For the induction we assume $k\geq 2$, so that
\begin{gather*}
F_1  F_2^k \hat{F}_3^l F_1^m\, v_\la =
F_1F_2^2  F_2^{k-2} \hat{F}_3^l F_1^m\, v_\la =
\bigl(-F_2^2F_1 +(q+q^{-1})F_2F_1F_2\bigr) F_2^{k-2} \hat{F}_3^l F_1^m\, v_\la
\end{gather*}
by the $q$-Serre relation \eqref{eq:qSerre}. Using the induction hypothesis, we find
\begin{multline*}
F_1  F_2^k \hat{F}_3^l F_1^m\, v_\la =
- F_2^2 \bigl( a_{k-2} F_2^{k-2} \hat{F}_3^l F_1^{m+1}\, v_\la +
b_{k-2} F_2^{k-3} \hat{F}_3^{l+1} F_1^{m}\, v_\la\bigr) \\
+ (q+q^{-1}) F_2  \bigl(
a_{k-1} F_2^{k-1} \hat{F}_3^l F_1^{m+1}\, v_\la +
b_{k-1} F_2^{k-2} \hat{F}_3^{l+1} F_1^{m}\, v_\la\bigr) \\
=(-a_{k-2}+(q+q^{-1})a_{k-1}) F_2^k \hat{F}_3^l F_1^{m+1}\, v_\la +
(-b_{k-2}+(q+q^{-1})b_{k-1})F_2^{k-1} \hat{F}_3^{l+1} F_1^{m}\, v_\la
\end{multline*}
which proves the induction step as well as the recurrence
\begin{gather*}
a_k + a_{k-2}=(q+q^{-1})a_{k-1}, \qquad
b_k + b_{k-2}=(q+q^{-1})b_{k-1}, \quad k\geq 2.
\end{gather*}
This recursion is solved by the Chebyshev polynomials (of the second kind) at $\frac12(q+q^{-1})$  as
well as by the associated Chebyshev polynomials. This gives the solution for
the recurrences and proves \eqref{prop:gen_3}

The action of $E_1$ follows from that of $F_1$, considering the adjoint. Note that
\begin{align*}
\langle E_1  F_2^k \hat{F}_3^l F_1^m\, v_\la , F_2^{k'} \hat{F}_3^{l'} F_1^{m'} \, v_\la \rangle
&= \langle F_2^k \hat{F}_3^l F_1^m\, v_\la , E_1^\ast F_2^{k'} \hat{F}_3^{l'} F_1^{m'} \, v_\la \rangle \\
&= \langle F_2^k \hat{F}_3^l F_1^m\, v_\la , K_1F_1  F_2^{k'} \hat{F}_3^{l'} F_1^{m'} \, v_\la \rangle,
\end{align*}
equals zero if $(k',l',m')\neq (k,l,m+1), (k+1,l-1,m)$. Moreover we have
\begin{align*}
\alpha_k(l,m)H_{k,l,m-1} &=\langle E_1 F_2^k\hat F_3^l F_1^m \, v_\lambda, F_2^{k}\hat F_3^l F_1^{m-1} \, v_\lambda\rangle \\
&=\langle F_2^k\hat F_3^l F_1^m \, v_\lambda, K_1F_1 F_2^{k}\hat F_3^l F_1^{m-1} \, v_\lambda\rangle \\
&= q^{k-l-2m+\lambda_1}\, a_k(l,m-1)\, H_{k,l,m},
\end{align*}
and
\begin{align*}
\beta_k(l,m)H_{k+1,l-1,m} &=\langle E_1 F_2^k\hat F_3^l F_1^m \, v_\lambda, F_2^{k+1}\hat F_3^{l-1} F_1^{m} \, v_\lambda\rangle \\
&=\langle F_2^k\hat F_3^l F_1^m \, v_\lambda, K_1F_1 F_2^{k-1}\hat F_3^{l-1} F_1^{m} \, v_\lambda\rangle \\
&= q^{k-l-2m+\lambda_1+2}\, b_{k+1}(l-1,m)\, H_{k,l,m}.
\end{align*}
Now the expressions of $\alpha_k(l,m)$ and $\beta_k(l,m)$ follow from the explicit expression of $H_{k,l,m}$ Theorem \ref{thm:basisVla} by a straightforward computation. \end{proof}


\section{The coideal subalgebra}
\label{sec:coideal-subalgebra}

In this section we follow Kolb \cite{Kol14} and introduce a right coideal subalgebra $\mathcal{B}$ of $\Uq$ which is the quantum analogue of $\mathcal{U}(\mathfrak{k})$ with $\mathfrak{k}=\mathfrak{u}(2)$ embedded in $\mathfrak{g}=\mathfrak{su}(3)$.  Let $c_1, c_2 \in \C^{\times}$ and write $c = (c_1, c_2)$.
Following \cite[Example 9.4]{Kol14}, $\mathcal{B}_{c} = \mathcal{B}$ is the right coideal subalgebra of $\Uq$, i.e. $\Delta(\mathcal{B}) \subset \mathcal{B} \tensor \Uq$, generated by
\begin{equation}
\label{uqsl3_eqn_B_generators}
K^{\pm 1} = \left(K_1 K_2^{-1}\right)^{\pm 1}, \quad B_1^{c} = B_1 = F_1 - c_1 E_2 K_1^{-1}, \quad B_2^{c} = B_2 = F_2 - c_2 E_1 K_2^{-1}.
\end{equation}
Throughout Sections \ref{sec:coideal-subalgebra} and \ref{sec:branching-rule} we omit the subscript and superscript $c$ in $\mathcal{B}_c$ and $B_i^c$  since the coideal subalgebra $\mathcal{B}$ will be fixed.

If we assume $c_1 \overline{c_2} = q^3 = \overline{c_1} c_2$ then it follows that $B_1^\ast = -\overline{c_1} K^{-1} B_2$, $B_2^\ast = -\overline{c_2} K B_1$ and $K^\ast = K$, so that $\mathcal{B}^\ast=\mathcal{B}$. By a straightforward computation we have
\begin{equation*}
\begin{split}
\Delta(B_1) &= B_1 \tensor K_1^{-1} + 1 \tensor F_1 - c_1 K^{-1} \tensor E_2 K_1^{-1}, \\
\Delta(B_2) &= B_2 \tensor K_2^{-1} + 1 \tensor F_2 - c_2 K \tensor E_1 K_2^{-1}.
\end{split}
\end{equation*}
The Serre relations for $\mathcal{B}$ follow from from \cite[Lemma 7.2, Theorem 7.4]{Kol14} taking $\mathcal{Z}_1 = -K^{-1}$ and $\mathcal{Z}_2 = -K$
\begin{equation}
\label{eq:serre_relations_B}
\begin{split}
B_1^2 B_2 - [2]_q B_1 B_2 B_1 + B_2 B_1^2 &= [2]_q (q c_2 K + q^{-2} c_1 K^{-1}) B_1, \\
B_2^2 B_1 - [2]_q B_2 B_1 B_2 + B_1 B_2^2 &= [2]_q (q c_1 K^{-1} + q^{-2} c_2 K) B_2.
\end{split}
\end{equation}
Alternatively \eqref{eq:serre_relations_B} can be verified directly from the definitions of $B_1$, $B_2$ and $K$.

The Cartan subalgebra of $\mathcal{B}$ is generated by $K^{\pm 1}$, $C_1$ and $C_2$, where
\begin{equation}
\label{uqsl3_eqn_C1C2}
\begin{split}
C_1 &= B_1 B_2 - q B_2 B_1 - \frac{1}{q - q^{-1}} c_2 K + \frac{q + q^{-1}}{q - q^{-1}} c_1 K^{-1}, \\
C_2 &= B_2 B_1 - q B_1 B_2 - \frac{1}{q - q^{-1}} c_1 K^{-1} + \frac{q + q^{-1}}{q - q^{-1}} c_2 K.
\end{split}
\end{equation}
Moreover if $c_1, c_2 \in \R^{\times}$, then $C_1$ and $C_2$ are self-adjoint.
The generators of  the Cartan subalgebra of $\mathcal{B}$ satisfy the relations $[K, C_i] = 0$ for $i = 1,2$, $[C_1, C_2] = 0$ and
\begin{equation}
\label{uqsl3_eqn_comm_relations_cartan}
\begin{split}
K B_1 &= q^{-3} B_1 K, \quad C_1 B_1 = q B_1 C_1, \quad C_2 B_1 = q^{-1} B_1 C_2, \\
K B_2 &= q^3 B_2 K, \quad C_1 B_2 = q^{-1} B_2 C_1, \quad C_2 B_2 = q B_2 C_2.
\end{split}
\end{equation}
Note that by \cite[Theorem 8.5]{KL08} the center of $\mathcal{B}$ is of rank $2$.
Hence the center of $\mathcal{B}$ is generated by $K^{\frac{1}{3}} C_1$ and $K^{-\frac{1}{3}} C_2$, extending $\mathcal{B}$ by cube roots of $K$. Then the central elements are self-adjoint for $c_1, c_2\in \R^{\times}$.

\subsection{Representation theory of $\mathcal{B}$}

Let $(\tau,W)$ be a finite-dimensional representation of $\mathcal{B}$. Since $W$ is a finite-dimensional complex vector space, there exists a vector $w\in W$ such that $\tau(K)w=\nu\, w$ for some $\nu\in \C$.  Then it follows from \eqref{uqsl3_eqn_comm_relations_cartan} that
$$\tau(K)\tau(B_1)^i \, w=q^{-3i}\, \tau(B_1)^i\tau(K)w=q^{-3i} \nu \tau(B_1)^i\, w,\qquad i\in \N,$$
so that the vectors $(\tau(B_1)^i \, w)_i$ are eigenvectors of $\tau(K)$ with different eigenvalues. Since $W$ is finite-dimensional, there exists $j\in \N$ such that $\tau(B_1^{j+1})w=0$ and $\tau(B_1^{j})w\neq 0$. Therefore $w_0=\tau(B_1^{j})w$ is a highest weight vector, i.e.
$$\tau(B_1)\, w_0=0, \qquad \tau(K)\, w_0 = \kappa \, w_0, \qquad q^{-3}\kappa \, \notin \sigma(K),$$
where $\kappa$ is the weight of $w_0$ and $\sigma(K)$ is the spectrum of $K$. Note that $\kappa \in \C^{\times}$ since it is the eigenvalue of an invertible operator.

\begin{prop} \label{uqsl3_prop_C1C2mu}
Let $\tau$ be a finite-dimensional irreducible representation of $\mathcal{B}$ on the vector space $W$. Then $\tau$ is determined by the dimension of $W$ and the action of $K$ on a highest weight vector.
\end{prop}

\begin{proof}
Let $\kappa \in \C^\times$ be the highest weight of $\tau$ and let $w_0$ be a highest weight vector, i.e. $\tau(K) \, w_0 = \kappa \,  w_0$ and $\tau(B_1)w_0=0$.  Since $\tau(K), \tau(C_1)$ and $\tau(C_2)$ form a commuting family of operators, we can assume that $\tau(C_1) \, w_0 = \eta_1 \, w_0$ and $\tau(C_2) \, w_0 = \eta_2 \, w_0$.
For every $i\in \N$, we define the vector $w_i = \tau(B_2)^i w_0 \in W$. Since $W$ is finite-dimensional, there exists $n \in \N$ such that $w_i \neq 0$ for $0 \leq i \leq n$ and $w_{n+1} = 0$. It follows from \eqref{uqsl3_eqn_comm_relations_cartan} that $\tau(K) w_i = q^{3i}\, \kappa \,  w_i$, so that $(w_i)_{i=0}^n$ is a set of linearly independent vectors since they are eigenvectors of $\tau(K)$ for different eigenvalues. Moreover \eqref{uqsl3_eqn_comm_relations_cartan}  implies
$$\tau(C_1)\, w_i = \tau(C_1) \tau(B_2)^i \, w_0 = q^{-i} \, \tau(B_2)^i \tau(C_1) \,w_0 = \eta_1q^{-i} \, w_i,$$
and similarly $\tau(C_2)\, w_i = \eta_2q^{i}\, w_i$. We will show that it is indeed a basis of $W$.

We prove by induction in $i$ that there exist $b_i\in \mathbb{C}$ such that $\tau(B_1) w_i = b_i w_{i-1}$ for $i=0,\ldots,n$. The statement holds for $i=0$ taking $b_0=0$ since $w_0$ is a highest weight vector. Let $i > 0$ and assume that $\tau(B_1) w_j = b_j w_{j-1}$ for  for all $j < i$. Using \eqref{uqsl3_eqn_C1C2} we find the recurrence relation
\begin{equation*}
\begin{split}
\tau(B_1) w_i &= \tau(B_1) \tau(B_2)^i w_0 = \tau(B_1B_2)w_{i+1} \\
  &= \tau \left( C_1 + q B_2 B_1 + \frac{c_2}{(q - q^{-1})} K - \frac{(q + q^{-1})}{(q - q^{-1})} c_1 K^{-1} \right)  w_{i-1} \\
  &= q \tau(B_2 B_1) w_{i-1} +\tau \left( C_1 + \frac{c_2}{(q - q^{-1})} K - \frac{(q + q^{-1})}{(q - q^{-1})} c_1 K^{-1} \right) w_{i-1},
\end{split}
\end{equation*}
By the inductive hypothesis, $\tau(B_2 B_1) w_{i-1}=b_{i-1} \tau(B_2)w_{i-2} =b_{i-1}w_{i-1}$, so that
\begin{equation} \label{uqsl3_eqn_rec_B1.2}
\begin{split}
\tau(B_1) w_i &= \left( q b_{i-1} + q^{1-i}\eta_1 + \frac{q^{3i - 3} \kappa\, c_2}{(q - q^{-1})} - \frac{(q + q^{-1})}{(q - q^{-1})} q^{3 - 3i} \kappa^{-1}\,  c_1 \right) w_{i-1}.
\end{split}
\end{equation}
Hence $\tau(B_1) w_i = b_i w_{i-1}$. Since $\tau$ is an irreducible representation we have that $W=\tau(\mathcal{B})\, w_0 =  \langle \{ w_0,w_1,\ldots,w_n\}\rangle$, and therefore $(w_i)_{i=0}^n$ is a basis of $W$. This completes the proof of the proposition.
\end{proof}
\begin{rmk}
\label{rmk:coefficients_bj_nozero}
Since we assume $(\tau,W)$ irreducible, the coefficients $b_i$ in the proof of Proposition \ref{uqsl3_prop_C1C2mu} are non-zero for $i=1,\ldots,n$. This follows from the fact that if $b_{i_0}=0$ for some $1\leq i_0\leq n$, then $\langle \{ w_{i_0},w_{i_0+1},\ldots,w_n\}\rangle$ is an invariant subspace and this contradicts the irreducibility of $\tau$.
\end{rmk}
\begin{cor}
\label{cor:rep_B_basisw}
Let $(\tau,W)$ be a finite-dimensional irreducible representation of $\mathcal{B}$ of dimension $n+1$ and highest weight $\kappa$. let $w_0$ be a highest weight vector and let $w_i=(B_2)^iw_0$ for $i=1,\ldots,n$. Then $(w_i)_{i=0}^n$ is a basis of $W$. The action of the generators of $\mathcal{B}$ on this basis is given by
\begin{equation*}
\tau(K) \, w_j= q^{3j} \,\kappa \, w_j, \qquad \tau(B_2) \, w_j = w_{j+1},\qquad  \tau(B_1) \, w_j=b_jw_{j-1}
\end{equation*}
where
$$b_0=0,\qquad b_j= c_1 \, \kappa^{-1} \, q^{-2n-1} \, [j]_q \, \frac{(1-q^{2n-2j+2})(1+c_2c_1^{-1}\kappa^2q^{2j+2n-1})}{(q - q^{-1})}.$$
Moreover, $\tau(C_1)\, w_j = q^{-j} \, \eta_1\, w_j$ and $\tau(C_2) \, w_j = q^j \, \eta_2\, w_j$, where
\begin{equation*}
\eta_1 = \frac{c_1 \, \kappa^{-1} \, q (1 + q^{-2n-2}) - c_2 \, \kappa \, q^{2n}}{q - q^{-1}}, \quad
\eta_2 = \frac{c_2 \, \kappa \, q^{ - 1} (1 + q^{2n + 2}) - c_1 \, \kappa^{-1} \, q^{- 2n}}{q - q^{-1}}.
\end{equation*}
\end{cor}
\begin{proof}
The fact that $(w_i)_{i=0}^n$ is a basis of $W$ and the action of $\tau(K)$ on $w_j$ follow directly from the proof of Proposition \ref{uqsl3_prop_C1C2mu}.
It is clear that $b_0=0$. We now show that
\begin{equation}
\label{eq:formula_bj_eta1}
b_j= [j]_q \left( \eta_1 + \frac{c_2 \, \kappa \, q^{2j - 2} - c_1 \, \kappa^{-1} \, q^{1 - 2j} (1 + q^{2j})}{(q - q^{-1})} \right),
\end{equation}
for all $j=1,\ldots,n$. We proceed by induction on $i$. If $i=1$, then the statement follows directly from \eqref{uqsl3_eqn_rec_B1.2}. Now we assume that \eqref{eq:formula_bj_eta1} is true for some $j$, $1<j\leq n$. Then it follows from \eqref{uqsl3_eqn_rec_B1.2}  and the inductive hypothesis that
\begin{multline*}
b_j = q \, [j-1]_q \left( \eta_1 + \frac{c_2\, \kappa \, q^{2j - 4} - c_1\, \kappa^{-1} \, q^{3 - 2j} (1 + q^{2j-2})}{(q - q^{-1})} \right) \\
+ q^{1-j}\eta_1 + \frac{q^{3j - 3} \kappa\, c_2}{(q - q^{-1})} - \frac{(q + q^{-1})}{(q - q^{-1})} q^{3 - 3j} \kappa^{-1}\,  c_1.
\end{multline*}
Now \eqref{eq:formula_bj_eta1} follows by a straightforward computation.

It follows from the proof of Proposition \ref{uqsl3_prop_C1C2mu} that $\tau(C_1)\, w_j = q^{-j} \, \eta_1\, w_j$ where $\eta_1$ is the eigenvalue for the highest weight vector $w_0$. From the construction of the vectors $w_i$ in Proposition \eqref{uqsl3_prop_C1C2mu}, it follows that $\tau(B_2)\, w_n=0$. Hence \eqref{uqsl3_eqn_C1C2} and \eqref{eq:formula_bj_eta1} yield
\begin{align*}
q^{-n}\eta_1\, w_n = \tau(C_1) \, w_n &= q\tau(B_2B_1) \, w_n -\frac{1}{q-q^{-1}} c_2\, \tau(K)w_n + \frac{q+q^{-1}}{q-q^{-1}} \,  c_1 \, \tau(K^{-1})w_n \\
&= -\frac{q^{n+1}-q^{-n+1}}{q-q^{-1}} \eta_1 -\frac{q^{n+1}-q^{-n+1}}{q-q^{-1}}\left( \frac{ c_2\, \kappa\, q^{2n-2} - c_2 \, \kappa^{-1} \, q^{1-2n}(1+q^{2n}) }{q-q^{-1}}\right) \\
& \qquad \qquad -\frac{c_2\,\kappa\, q^{3n}}{q-q^{-1}} + \frac{(q+q^{-1})\, c_1\, \kappa^{-1} \, q^{-3n}}{q-q^{-1}}.
\end{align*}
Now the expression of $\eta_1$ follows by a straightforward computation. The expression of $\eta_2$ can be obtained similarly from the action of $C_2$ on $w_n$.
\end{proof}
\begin{rmk}
\label{rmk:tau_irred_kappa_neq}
If $\tau$ is an irreducible representation with highest weight $\kappa$ and dimension $n+1$, it follows from Remark \ref{rmk:coefficients_bj_nozero}  and the explicit expression of the coefficient $b_i$ in Corollary \ref{cor:rep_B_basisw}  that $c_2c_1^{-1}\kappa^2\neq -q^{-2j-2n+1}$ for all $j=1,\ldots,n$.
\end{rmk}

\begin{rmk}
\label{rmk:tau_determined_kappa_eta_1}
It follows from Proposition \ref{uqsl3_prop_C1C2mu} and Corollary \ref{cor:rep_B_basisw} that a finite-dimensional irreducible representation $(\tau,W)$ of $\mathcal{B}$ is completely determined by the highest weight $\kappa$ and the eigenvalue of $\eta_1$ of the highest weight vector as eigenvector of $\tau(C_1)$.
\end{rmk}

\begin{cor}
\label{cor:caracterizations_irreps_B}
Every irreducible finite-dimensional representation of $\mathcal{B}$ is determined by a pair $(\kappa,n)$ where $\kappa$ is the highest weight and the dimension is $n+1$. Conversely, to each pair $(\kappa,n)$ with $\kappa\in \mathbb{C}^\times$, $n\in \N$ and $\kappa^2 \notin -c_1c_2^{-1}q^{1-\N}$,  there corresponds an irreducible representation $(\tau_{(\kappa,n)},W_{(\kappa,n)})$ with highest weight $\kappa$ and dimension $n+1$.
\end{cor}
\begin{proof}
It follows directly from Proposition \ref{uqsl3_prop_C1C2mu},  Corollary \ref{cor:rep_B_basisw} and Remark \ref{rmk:tau_irred_kappa_neq}.
\end{proof}

\begin{prop}
Assume that $\kappa\in\mathbb{R}^{\times}$ and $c_1\overline c_2=q^3$. Let $(\tau,W)$ be an irreducible finite-dimensional representation of $\mathcal{B}$. Then $\tau$ is unitarizable.
\end{prop}
\begin{proof}
Since $c_1 \overline{c_2} = q^3$, we have that $\mathcal{B}^\ast=\mathcal{B}$. More precisely $B_1^\ast = -\overline{c_1} K^{-1} B_2$, $B_2^\ast = -\overline{c_2} K B_1$ and $K^\ast = K$. Let $(w_i)_{i = 0}^n$ be the basis of $W$ given in Corollary \ref{cor:rep_B_basisw} and let $\langle \cdot , \cdot \rangle$ be the hermitian bilinear form defined on the basis elements by $\langle w_0,w_0\rangle = 1$,
$$\langle w_k, w_k \rangle = \langle \tau(((B_2)^k)^\ast (B_2)^k) w_0, w_0 \rangle, \qquad \langle w_i,w_j \rangle = 0, \quad i\neq j.$$
Observe that
\begin{align}
\langle w_k, w_k \rangle &= \langle \tau(((B_2)^k)^\ast (B_2)^k) w_0, w_0 \rangle \nonumber \\
  &= (-1)^k \overline{c_2}^k q^{3\binom{k}{2}} \langle \tau(K^k (B_1)^k (B_2)^k) w_0, w_0 \rangle = (-1)^k \overline{c_2}^k q^{3\binom{k}{2}} \, \langle \tau(K^k)  w_0, w_0 \rangle \, \prod_{i=1}^k b_i \nonumber \\
  \label{eq:norm_wk}
  &=\frac{q^{3\binom{k}{2} -k(2n - 1)} }{(1-q^2)^k}\, [k]_q! \, (q^{2n};q^{-2})_k \, (-c_2c_1^{-1}\, \kappa^2\, q^{2n-1};q^2)_k \, \langle w_0,w_0\rangle.
\end{align}
Since $q^3 c_2 c_1^{-1} = c_1 \overline{c_2} c_2 c_1^{-1} = |c_2|^2 > 0$, it follows that $c_2 c_1^{-1} > 0$ and thus \eqref{eq:norm_wk} is positive. Therefore $\langle \cdot, \cdot \rangle$ is a positive definite bilinear form. Moreover,  $\langle \tau(X) \,  w_i , w_j \rangle  = \langle w_i, \tau(X^\ast) \, w_j \rangle$ for all $X\in \mathcal{B}$. This follows from a straightforward verification on the generators of $\mathcal{B}$.
\end{proof}

\begin{rmk}
Let $\kappa \in \R^{\times}$ and $n \in \N$.
Let $(w_i)_{i = 0}^n$ be the orthogonal basis for $W^{(\mu,n)}$ as in Corollary \ref{cor:rep_B_basisw}.
We define an orthonormal basis $(\widetilde{w}_i)_{i=0}^n$ by $\widetilde{w}_i = w_i / ||w_i||$.
The actions of $C_1$, $C_2$ and $K$ on the orthonormal basis are the same. For $B_1$ and $B_2$ we have
\begin{equation*}
\begin{split}
\tau_{(\kappa,n)}(B_1) \widetilde{w}_{i} &= -c_1 \,\kappa^{-1} \, q^{- 2i - n + 1} \sqrt{
  \frac{(1 - q^{2i})}{(1 - q^2)}
  \frac{(1 - q^{2n - 2i + 2})}{(1 - q^2)}
  (q + c_2 c_1^{-1} \, \kappa^2 \, q^{2n + 2i})
} \widetilde{w}_{i-1}, \\
\tau_{(\kappa,n)}(B_2) \widetilde{w}_{i} &= q^{i-n+1} \sqrt{
  \frac{(1 - q^{2i + 2})}{(1 - q^2)}
  \frac{(1 - q^{2n - 2i})}{(1 - q^2)}
  (q + c_2 c_1^{-1} \, \kappa^2 \,q^{ 2n + 2i + 2})
} \widetilde{w}_{i+1}.
\end{split}
\end{equation*}
\end{rmk}

\section{The branching rule}
\label{sec:branching-rule}

In this section we prove the main theorem of the paper. We fix a coideal subalgebra $\mathcal{B}$ and show that any finite-dimensional representation of $\Uq$ restricted to $\mathcal{B}$ decomposes multiplicity free as finite-dimensional representations of $\mathcal{B}$ and we characterise the $\mathcal{B}$-representations that occur in this decomposition.
In case $\mathcal{B}$ is $\ast$-invariant, every finite-dimensional irreducible representation of $\Uq$ restricted to $\mathcal{B}$ obviously decomposes into finite-dimensional irreducible representations. This fact is also noted by Letzter \cite[Theorem 3.3]{Let00}.

\begin{thm} \label{thm:branching}
Let $\lambda \in P^+$ such that $\lambda = \lambda_1 \varpi_1 + \lambda_2 \varpi_2$ and fix the finite-dimensional irreducible representation $\pi_{\lambda}$ of $\Uq$ on the vector space $V_{\lambda}$. Let $\mathcal{B}$ be a coideal subalgebra with $c_2c_1^{-1}\notin -q^{2\lambda_1+2\lambda_2+1-\mathbb{N}}$. The representation $\pi_{\lambda}$ restricted to $\mathcal{B}$ decomposes multiplicity free into irreducible representations;
\begin{equation*}
\pi_{\lambda} |_{\mathcal{B}} \simeq \bigoplus_{(\kappa, n)} \tau_{(\kappa, n)}, \qquad V_{\lambda} = \bigoplus_{(\kappa, n)} W_{(\kappa, n)},
\end{equation*}
where the sum is taken over $(\kappa, n) = (q^{\lambda_1 -\lambda_2-3i},i+x)$, with $0\leq i \leq \lambda_1$ and $0\leq x \leq \lambda_2$.
\end{thm}

The proof of Theorem  \ref{thm:branching} will be carried out in the next subsections. If $(\tau_{(\kappa,n)},W_{(\kappa,n)})$ is a representation of $\mathcal{B}$ that occurs in the representation $\pi_\lambda$ upon restriction to $\mathcal{B}$ then a highest weight vector $w_0^{(\mu,n)}$ for $\tau_{(\kappa,n)}$ is completely determined by the highest weight $\kappa$ and the eigenvalue $\eta_1$,
see Remark \ref{rmk:tau_determined_kappa_eta_1}. Hence, highest weight vectors for $\mathcal{B}$-representations in $V_\lambda$ are the eigenvectors of $\pi_\lambda(C_1)$ that belong to the kernel of $\pi_\lambda(B_1)$. In Subsection \ref{subsec:KerB1} we determine the kernel of $\pi_\lambda(B_1)$.
\begin{rmk}
\label{rmk_Serre_imply_ker}
Observe that the Serre relations \eqref{eq:serre_relations_B}  for $\mathcal{B}$ imply that
the kernel of $\pi_\lambda(B_1)$ is invariant under the action of $B_1B_2$ and thus under the action of $C_1$.
\end{rmk}
In Subsection \ref{subsec:diagC_1} we diagonalize the restriction of $\pi_\lambda(C_1)$ to $\ker(\pi_\lambda(B_1))$. In most of the proofs we identify $\pi_\lambda(X)$, $X\in \Uq$, with $X$.

\begin{rmk}
The restriction on $c_1$ and $c_2$ in Theorem \ref{thm:branching} is assumed in order to ensure the complete reducibility of $\pi_\lambda$ upon restriction to $\mathcal{B}$. This is not always true for the excluded values of $c_1$ and $c_2$. For example let $\lambda=\varpi_1$. Then $V_\lambda$ is a three dimensional vector space. Mudrov's basis in Theorem \ref{thm:basisVla} is given by
$$\mathscr{B}=\{ v_\lambda, \, F_1\, v_\lambda,\, F_2F_1\, v_\lambda  \}.$$
In this basis, the operator $C_1$ is given by the $3\times 3$ matrix
$$C_1=\begin{pmatrix} \frac{c_1q^2+c_1-qc_2)}{q(q^2-1)} & 0  & -c_1c_2 \\  0 &  \frac{x_1q^4+c_1-qc_2}{q^2-1} & 0 \\
-q &  0 & \frac{c_1q^4+c_1-q^3c_2}{q(q^2-1)} \end{pmatrix} .$$
The eigenvectors of $C_1$ are (multiples of) the vectors
$$\rho_1=\begin{pmatrix} c_1 \\ 0 \\ 1 \end{pmatrix}, \quad \rho_2 = \begin{pmatrix} 0 \\ 1 \\ 0 \end{pmatrix}, \quad
\rho_3=\begin{pmatrix} -c_2/q \\ 0 \\ 1 \end{pmatrix}.$$
If $c_1\neq -c_2/q$, then $V_\lambda$ decomposes as a sum of a two-dimensional and a one-dimensional irreducible representations of $W$:
$$V_\lambda= W_{(q,0)}\oplus W_{(q^{-2},1)},$$
where $W_{(q,0)}=\langle \{\rho_1\}\rangle$ and $W_{(q^{-2},1)}=\langle \{\rho_2,\rho_3\}\rangle$. Moreover, the highest weight vectors of $W_{(q,0)}$ and $W_{(q^{-2},1)}$ are $\rho_1$ and $\rho_2$ respectively. If we let $c_1=-c_2/q$ then the matrix $C_1$ degenerates into a non-diagonalizable matrix. The only eigenvectors are the multiples of $\rho_2$ and $\rho_3$ and therefore, although $W_{(q^{-2},1)}$ is a $\mathcal{B}$-invariant subspace of $V_\lambda$, there is no one-dimensional $\mathcal{B}$-invariant subspace in $V_\lambda$.
\end{rmk}


\subsection{The kernel of $B_1$}
\label{subsec:KerB1}

\begin{figure}[t!]
\begin{center}
\scalebox{.85}{
\begin{tikzpicture}

\coordinate (origin) at (0,0);
\coordinate (varpi1) at (1,0);
\coordinate (varpi2) at (60:1);
\coordinate (alpha1) at ($2*(varpi1) - (varpi2)$);
\coordinate (alpha2) at ($-1*(varpi1) + 2*(varpi2)$);
\coordinate (sph)    at ($(varpi1) + (varpi2)$);

\draw[->] (origin) -- (varpi1) node[above] {$\varpi_1$};
\draw[->] (origin) -- (varpi2) node[above] {$\varpi_2$};
\draw[->] (origin) -- (alpha1) node[right] {$\alpha_1$};
\draw[->] (origin) -- (alpha2) node[above] {$\alpha_2$};
\draw[->] (origin) -- (sph)    node[right] {$\alpha_1 + \alpha_2$};

\draw[gray, dashed] (origin) -- ($6*(varpi1)$)
  (origin) -- ($7*(varpi2)$);

\foreach \i in {0,...,2} {
  \node[circle, fill=black, scale=0.3] at ($3*(varpi2) - \i*(sph)$) {};
  \node[circle, fill=black, scale=0.3] at ($-3*(varpi1) + \i*(alpha1)$) {};
  \node[circle, fill=black, scale=0.3] at ($3*(varpi1) - 3*(varpi2) + \i*(alpha2)$) {};
}

\draw[-] ($3*(varpi2)$) -- ($-3*(varpi1)$)
  ($-3*(varpi1)$) -- ($3*(varpi1) - 3*(varpi2)$)
  ($3*(varpi1) - 3*(varpi2)$) -- ($3*(varpi2)$);

\foreach \i in {0,...,4} {
  \node[circle, fill=black, scale=0.3] at ($-1*(varpi1) + 5*(varpi2) - \i*(sph)$) {};
  \node[circle, fill=black, scale=0.3] at ($-4*(varpi1) - (varpi2) + \i*(alpha1)$) {};
  \node[circle, fill=black, scale=0.3] at ($5*(varpi1) - 4*(varpi2) + \i*(alpha2)$) {};
}

\draw[-] ($-1*(varpi1) + 5*(varpi2)$) -- ($-5*(varpi1) + (varpi2)$)
  ($-5*(varpi1) + (varpi2)$) -- ($-4*(varpi1) - (varpi2)$)
  ($-4*(varpi1) - (varpi2)$) -- ($4*(varpi1) - 5*(varpi2)$)
  ($4*(varpi1) - 5*(varpi2)$) -- ($5*(varpi1) - 4*(varpi2)$)
  ($5*(varpi1) - 4*(varpi2)$) -- ($(varpi1) + 4*(varpi2)$)
  ($(varpi1) + 4*(varpi2)$) -- ($-1*(varpi1) + 5*(varpi2)$);

\foreach \i in {0,...,5} {
  \node[circle, fill=black, scale=0.3] at ($-2*(varpi1) + 7*(varpi2) - \i*(sph)$) {};
  \node[circle, fill=black, scale=0.3] at ($-5*(varpi1) - 2*(varpi2) + \i*(alpha1)$) {};
  \node[circle, fill=black, scale=0.3] at ($7*(varpi1) - 5*(varpi2) + \i*(alpha2)$) {};
}

\node[circle, fill=black, scale=0.3] at ($6*(varpi2)$) {};
\node[circle, fill=black, scale=0.3] at ($6*(varpi1) - 6*(varpi2)$) {};
\node[circle, fill=black, scale=0.3] at ($-6*(varpi1)$) {};

\draw[-] ($-2*(varpi1) + 7*(varpi2)$) -- ($-7*(varpi1) + 2*(varpi2)$)
  ($-7*(varpi1) + 2*(varpi2)$) -- ($-5*(varpi1) - 2*(varpi2)$)
  ($-5*(varpi1) - 2*(varpi2)$) -- ($5*(varpi1) - 7*(varpi2)$)
  ($5*(varpi1) - 7*(varpi2)$) -- ($7*(varpi1) - 5*(varpi2)$)
  ($7*(varpi1) - 5*(varpi2)$) -- ($2*(varpi1) + 5*(varpi2)$)
  ($2*(varpi1) + 5*(varpi2)$) -- ($-2*(varpi1) + 7*(varpi2)$);

\node[right] at ($2*(varpi1) + 5*(varpi2)$) {$\lambda = \lambda_1 \varpi_1 + \lambda_2 \varpi_2$};

\draw[thick, ->, shorten <= 0.2cm, shorten >= 0.2cm]
  ($2*(alpha2)$) -- ($2*(alpha2)-(alpha1)$);
\draw[thick, ->, shorten <= 0.2cm, shorten >= 0.2cm]
  ($2*(alpha2)$) -- ($3*(alpha2)$);
\node at ($2*(alpha2)+ 0.2*(alpha2)-0.2*(alpha1)$) {$B_1$};

\draw[thick, ->, shorten <= 0.2cm, shorten >= 0.2cm]
  ($-5*(varpi1) + 4*(varpi2)$) -- ($-3*(varpi1) + 3*(varpi2)$);
\draw[thick, ->, shorten <= 0.2cm, shorten >= 0.2cm]
  ($-5*(varpi1) + 4*(varpi2)$) -- ($-4*(varpi1) + 2*(varpi2)$);
\node at ($-4*(varpi1) + 3*(varpi2)$) {$B_2$};

\fill[gray,rounded corners=2mm,opacity=0.45]   ($-6.4*(varpi1)+0.4*(varpi1)-0.4*(varpi2)$) --  ($-6.4*(varpi1)$)  -- ($6.4*(varpi2)$) -- ($6.4*(varpi2)+0.4*(varpi1)-0.4*(varpi2)$) -- cycle;

\fill[gray,rounded corners=2mm,opacity=0.45]   ($-6.4*(varpi1)+(alpha2)+0.4*(varpi1)-0.4*(varpi2)$) --  ($-6.4*(varpi1)+(alpha2)$)  -- ($6.4*(varpi2)-(alpha1)$) -- ($6.4*(varpi2)-(alpha1)+0.4*(varpi1)-0.4*(varpi2)$) -- cycle;

\fill[gray,rounded corners=2mm,opacity=0.45]   ($-6.4*(varpi1)-(alpha2)+0.4*(varpi1)-0.4*(varpi2)$) --  ($-6.4*(varpi1)-(alpha2)$)  -- ($6.4*(varpi2)+(alpha1)$) -- ($6.4*(varpi2)+(alpha1)+0.4*(varpi1)-0.4*(varpi2)$) -- cycle;

\node[left] at ($-6.2*(varpi1)+(alpha2)$) {$U_2$};

\node[left] at ($-6.2*(varpi1)$) {$U_1$};

\node[left] at ($-6.2*(varpi1)-(alpha2)$) {$U_0$};

\end{tikzpicture}}
\end{center}
\caption{Weight diagram for the weight $\lambda = 2\varpi_1 + 5\varpi_2$. The subspaces $U_i$ defined in \eqref{eq:definition_U_i} are spanned by the basis vectors indicated in gray.}
\label{uqsl3_fig_weightdiagram}
\end{figure}
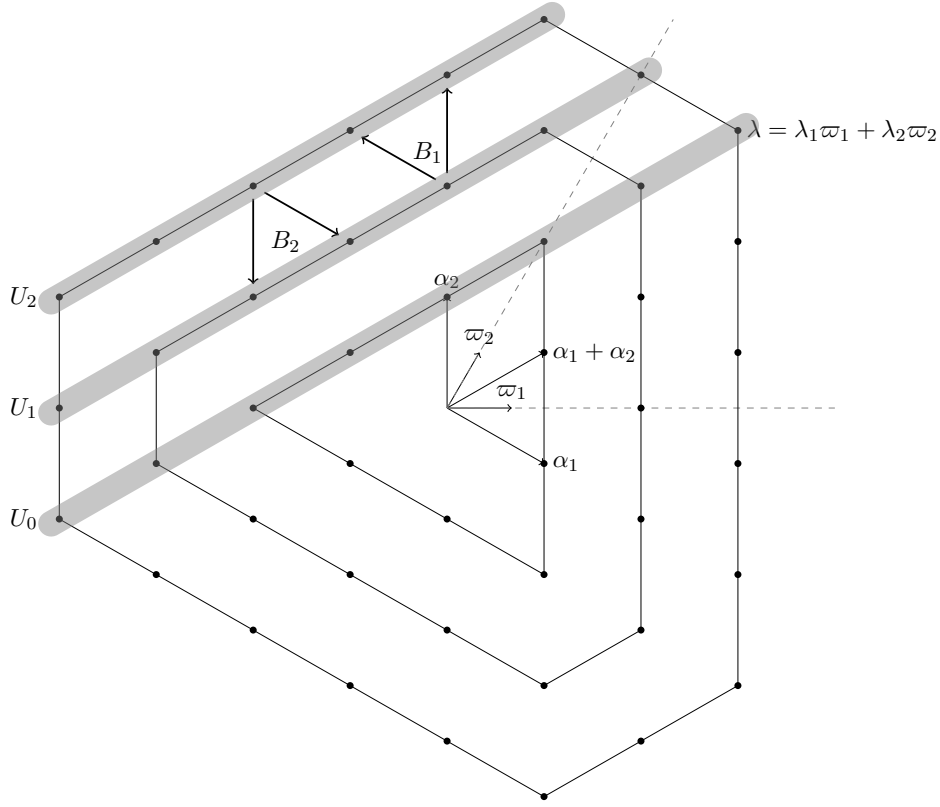
The goal of this subsection is to describe the structure of the kernel of $\pi_\lambda(B_1)$ by introducing a particular basis. For each $i=0,\ldots,\lambda_1$, we introduce the following subspaces of $V_\lambda$:
\begin{equation}
\label{eq:definition_U_i}
U_i=\langle \mathscr{B}_i \rangle, \qquad \mathscr{B}_i= \{ F_2^k\hat F_3^l F_1^{i+k} \, v_\lambda : \, 0\leq l \leq \lambda_2,\,\, 0\leq k \leq \lambda_1-i \}.
\end{equation}
It follows from weight space considerations, that $F_1,E_2:U_i\to U_{i+1}$ and $F_2,E_1:U_{i+1}\to U_{i}$ so that $B_1:U_{i}\to U_{i+1}$ and $B_2:U_{i+1}\to U_i$. This is shown in Figure \ref{uqsl3_fig_weightdiagram} for the highest weight $\lambda = 2\varpi_1 + 5\varpi_2$.

\begin{rmk}
\label{rmk:layer_structure}
For each $i=0,\ldots,\lambda_1$, the basis $\mathscr{B}_i$ consists on $\lambda_1-i+1$ layers of $\lambda_2+1$ vectors. More precisely, for $k=0,\ldots, \lambda_1-i$, the $k$-th layer is given by the vectors
$$F_2^k\hat F_3^l F_1^{k+i}\, v_\lambda,\qquad l=0,\ldots,\lambda_2.$$
This structure is indicated in the Figure \ref{fig:structure_basis_Ui} for the representation  $\lambda = 2\varpi_1 + 5\varpi_2$. The layers appear as circled numbers.
\end{rmk}

\begin{rmk}
The dimension of $U_i$ is $(\lambda_2+1)(\lambda_1-i+1)$. Therefore, the dimension of $\ker (B_1){|_{U_i}}$ is, at least, $\lambda_2+1$. In particular, $U_{\lambda_1}\subset \ker(B_1)$.
\end{rmk}

\begin{prop}
\label{prop:kernel_B1}
The kernel of $\pi_\lambda(B_1){|_{U_i}}$ has dimension $\lambda_2+1$. Moreover, a basis of $\ker \pi_\lambda(B_1){|_{U_i}}$ is given by $(u^i_n)_{n=0}^{\lambda_2}$, where
$$u^i_{n}=\sum_{k=0}^{\lambda_1-i} \sum_{l=0}^{\lambda_2} \gamma^n_{k,l} \, F_2^k\hat F_3^l F_1^{k+i}\, v_\lambda,$$
and the coefficients $\gamma^n_{k,l}$ are given by the recurrence relation
\begin{multline*}
a_k(l,k+i) \, \gamma^n_{k,l}  + b_{k+1}(l-1,k+i+1)  \, \gamma^n_{k+1,l-1}  \\
- c_1 \, q^{l+2i+k+1-\lambda_1} \, \eta_{k+1}(l,k+i+1) \, \gamma^n_{k+1,l} = 0,
\end{multline*}
for $k=1,\ldots,\lambda_1-i-1$, $l=0,\ldots,\lambda_2$, with initial values $\gamma^n_{\lambda_1-i,l}=\delta_{n,l}$.
\end{prop}
\begin{proof}
Let $u=\sum_{k=0}^{\lambda_1-i} \sum_{l=0}^{\lambda_2} \gamma_{k,l} \, F_2^k\hat F_3^l F_1^{k+i}\, v_\lambda$ be a vector in the kernel of $B_1$. Then
\begin{align*}
B_1  u &= \sum_{k=0}^{\lambda_1-i} \sum_{l=0}^{\lambda_2} \gamma_{k,l} \,  (F_1  F_2^k\hat F_3^l F_1^{k+i}\, v_\lambda - c_1E_2K_1^{-1}  F_2^k\hat F_3^l F_1^{k+i}\, v_\lambda) \displaybreak[0] \\
& = \sum_{k=0}^{\lambda_1-i} \sum_{l=0}^{\lambda_2} \gamma_{k,l} \, ( a_k(l,k+i) \, F_2^k\hat F_3^l F_1^{k+i+1} \, v_\lambda + b_k(l,k+i) \, F_2^{k-1}\hat F_3^{l+1} F_1^{k+i} \, v_\lambda \\
& \hspace{6.5cm}- c_1 q^{l+2i+k-\lambda_1} \eta_k(l,k+i) \, F_2^{k-1}\hat F_3^l F_1^{k+i}\, v_\lambda) \displaybreak[0] \\
& = \sum_{k=0}^{\lambda_1-i} \sum_{l=0}^{\lambda_2} ( a_k(l,k+i) \,\gamma_{k,l} +  b_{k+1}(l-1,k+i+1) \, \gamma_{k+1,l-1}  \\
& \hspace{4.5cm} -c_1\, q^{l+2i+k+1-\lambda_1} \, \eta_{k+1}(l,k+i+1)\gamma_{k+1,l} ) \, F_2^k\hat F_3^l F_1^{k+i+1}  \, v_\lambda.
\end{align*}
Since the elements $F_2^k\hat F_3^l F_1^{k+i+1}$, $0\leq k \leq \lambda_1-i$, $0\leq l \leq \lambda_2$, are linearly independent it follows that the coefficients $\gamma_{k,l}$ satisfy the following recurrence relation.
\begin{equation}
\label{eq:rec_relation_gammas}
a_k(l,k+i) \,\gamma_{k,l} +  b_{k+1}(l-1,k+i+1) \, \gamma_{k+1,l-1}
 -c_1\, q^{l+2i+k+1-\lambda_1} \, \eta_{k+1}(l,k+i+1) \, \gamma_{k+1,l}= 0.
 \end{equation}
For each $n=0,1,\ldots, \lambda_2$, if we set $\gamma^n_{\lambda_1-i,l}=\delta_{n,l}$, then \eqref{eq:rec_relation_gammas} determines uniquely a vector $u_n$ in the kernel of $B_1$. The vectors  $u_n$ are clearly linearly independent and span the kernel of $B_1$ restricted to $U_i$. This completes the proof of the proposition.
\end{proof}
\begin{figure}[t!]
\begin{center}
\scalebox{.95}{
\begin{tikzpicture}

\draw (0,0) circle (.1cm) node[label={[label distance=.1cm,scale=.7]90:$ v_\lambda$}] {};
\draw[shorten >= 5pt, shorten <= 5pt] (0,0) to (2,0);

\draw (2,0) circle (.1cm) node[label={[label distance=.1cm,scale=.7]90:$\hat F_3  v_\lambda$}] {};
\draw[shorten >= 5pt, shorten <= 5pt] (2,0) to (4,0);

\draw (4,0) circle (.1cm) node[label={[label distance=.1cm,scale=.7]90:$\hat F_3^2  v_\lambda$}] {};
\draw[shorten >= 5pt, shorten <= 5pt] (4,0) to (6,0);

\draw (6,0) circle (.1cm) node[label={[label distance=.1cm,scale=.7]90:$\hat F_3^3  v_\lambda$}] {};
\draw[shorten >= 5pt, shorten <= 5pt] (6,0) to (8,0);

\draw (8,0) circle (.1cm) node[label={[label distance=.1cm,scale=.7]90:$\hat F_3^4  v_\lambda$}] {};
\draw[shorten >= 5pt, shorten <= 5pt] (8,0) to (10,0);

\draw (10,0) circle (.1cm) node[label={[label distance=.1cm,scale=.7]90:$\hat F_3^5 v_\lambda$}] {};

\draw (2,1) circle (.1cm) node[label={[label distance=.1cm,scale=.7]90:$F_2F_1v_\lambda$}] {};
\draw[shorten >= 5pt, shorten <= 5pt] (2,1) to (4,1);

\draw (4,1) circle (.1cm) node[label={[label distance=.1cm,scale=.7]90:$F_2\hat F_3 F_1 v_\lambda$}] {};
\draw[shorten >= 5pt, shorten <= 5pt] (4,1) to (6,1);

\draw (6,1) circle (.1cm) node[label={[label distance=.1cm,scale=.7]90:$F_2\hat F_3^2 F_1 v_\lambda$}] {};
\draw[shorten >= 5pt, shorten <= 5pt] (6,1) to (8,1);

\draw (8,1) circle (.1cm) node[label={[label distance=.1cm,scale=.7]90:$F_2\hat F_3^3 F_1 v_\lambda$}] {};
\draw[shorten >= 5pt, shorten <= 5pt] (8,1) to (10,1);

\draw (10,1) circle (.1cm) node[label={[label distance=.1cm,scale=.7]90:$F_2\hat F_3^4 F_1 v_\lambda$}] {};
\draw[shorten >= 5pt, shorten <= 5pt ] (10,1) to (12,1);

\draw (12,1) circle (.1cm) node[label={[label distance=.1cm,scale=.7]90:$F_2\hat F_3^5 F_1 v_\lambda$}] {};

\draw (4,2) circle (.1cm) node[label={[label distance=.1cm,scale=.7]90:$F_2^2\hat F^2_1 v_\lambda$}] {};
\draw[shorten >= 5pt, shorten <= 5pt] (4,2) to (6,2);

\draw (6,2) circle (.1cm) node[label={[label distance=.1cm,scale=.7]90:$F_2^2\hat F_3 F_1^2 v_\lambda$}] {};
\draw[shorten >= 5pt, shorten <= 5pt] (6,2) to (8,2);

\draw (8,2) circle (.1cm) node[label={[label distance=.1cm,scale=.7]90:$F_2^2\hat F_3^2F_1^2 v_\lambda$}] {};
\draw[shorten >= 5pt, shorten <= 5pt] (8,2) to (10,2);

\draw (10,2) circle (.1cm) node[label={[label distance=.1cm,scale=.7]90:$F_2^2\hat F_3^3 F_1^2 v_\lambda$}] {};
\draw[shorten >= 5pt, shorten <= 5pt] (10,2) to (12,2);

\draw (12,2) circle (.1cm) node[label={[label distance=.1cm,scale=.7]90:$F_2^2\hat F_3^4 F_1^2 v_\lambda$}] {};;
\draw[shorten >= 5pt, shorten <= 5pt ] (12,2) to (14,2);

\draw (14,2) circle (.1cm) node[label={[label distance=.1cm,scale=.7]90:$F_2^2\hat F_3^5 F_1^2 v_\lambda$}] {};;

\node[left,scale=.7] at (-0.5,0) {\circled{0}};
\node[left,scale=.7] at (-0.5,1) {\circled{1}};
\node[left,scale=.7] at (-0.5,2) {\circled{2}};

\end{tikzpicture}}
\caption{Structure of the basis of $U_0$ for the representation $\pi_\lambda$ with $\lambda = 2\varpi_1 + 5\varpi_2$ as in Figure \ref{uqsl3_fig_weightdiagram}. The circled numbers indicate the layers of the basis.}
\label{fig:structure_basis_Ui}
\end{center}
\end{figure}
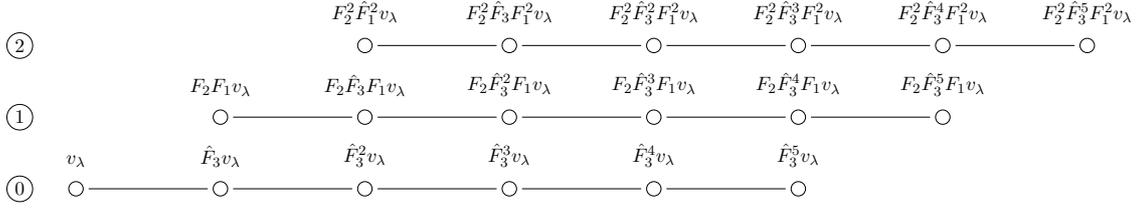

\begin{rmk}
\label{rmk:form_ui}
According to the layer structure of $\mathscr{B}_i$ described in Remark \ref{rmk:layer_structure}, the vector $u_n^i$ has a single non-zero contribution from the  vectors of the upper layer, namely from  $F_2^{\lambda_1-i} \hat F_3^n F_1^{\lambda_1}$,  and two contributions from the one but upper layer. Therefore, we have
\begin{multline}
u^i_n=F_2^{\lambda_1-i}\hat F_3^nF_1^{\lambda_1}\, v_\lambda + \gamma^n_{\lambda_1-i-1,n}   \,  F_2^{\lambda_1-i-1} \hat F_3^n F_1^{\lambda_1-1}\, v_\lambda+ \gamma^n_{\lambda_1-i-1,n+1}   \,  F_2^{\lambda_1-i-1} \hat F_3^{n-1} F_1^{\lambda_1-1} \, v_\lambda\\
+ \sum_{k=0}^{\lambda_1-i-2} \sum_{l=0}^{\lambda_2} \gamma^n_{k,l} \, F_2^k \hat F_3^l F_1^{i+k}\, v_\lambda.
\end{multline}
The coefficients $\gamma^n_{\lambda_1-i-1,n}$ and $\gamma^n_{\lambda_1-i-1,n+1}$ corresponding to the vectors of the one but last layer are given by
\begin{equation}
\label{eq:expression_gammas}
\begin{split}
\gamma^n_{\lambda_1-i-1,n} &= \frac{c_1\, q^{n+i} \, (q^{\lambda_1-i}-q^{-\lambda_1+i})(q^{\lambda_2+\lambda_1-n}-q^{-\lambda_2-\lambda_1+n})}{(q-q^{-1})^2},\\
\gamma^n_{\lambda_1-i-1,n+1} &=- \frac{(q^{\lambda_1-i}-q^{-\lambda_1+i})(q^{\lambda_2+\lambda_1-n-1}-q^{-\lambda_2-\lambda_1+n+1})}{(q^{\lambda_2+\lambda_1+1-n}-q^{-\lambda_2-\lambda_1-1+n})(q^{\lambda_2+i-n}-q^{-\lambda_2-i+n})}.
\end{split}
\end{equation}
The structure of the vectors $u_n^i$ for $U_{\lambda_1-2}$ is depicted in Figure \ref{fig:structure_ui}.
\end{rmk}
\begin{rmk}
\label{rmk:ui_not_orthogonal}
The basis $\{u^i_n\}^i_n$ of the kernel of $\pi_\lambda(B_1)$ is not an orthogonal basis. In fact, it follows from Remark \ref{rmk:form_ui} that
\begin{align*}
u_0^{\lambda_1-1} &= F_2^{\lambda_1-1}F_1^{\lambda_1}\, v_\lambda + \gamma^0_{\lambda_1-2,0}   \,  F_2^{\lambda_1-2} F_1^{\lambda_1-1}\, v_\lambda, \\
u_1^{\lambda_1-1} &= F_2^{\lambda_1-1}\hat F_3F_1^{\lambda_1}\, v_\lambda + \gamma^1_{\lambda_1-2,1}   \,  F_2^{\lambda_1-2} \hat F_3 F_1^{\lambda_1-1}\, v_\lambda+ \gamma^1_{\lambda_1-2,2}   \,  F_2^{\lambda_1-2} F_1^{\lambda_1-1}\, v_\lambda,
\end{align*}
and therefore
$$\langle u_0^{\lambda_1-1},u_1^{\lambda_1-1}\rangle=  \gamma^0_{\lambda_1-2,0}\gamma^1_{\lambda_1-2,2} H_{\lambda_1-2,0,\lambda_1-1}\neq0,$$
using the explicit expressions \eqref{eq:expression_gammas}.
\end{rmk}

\begin{center}
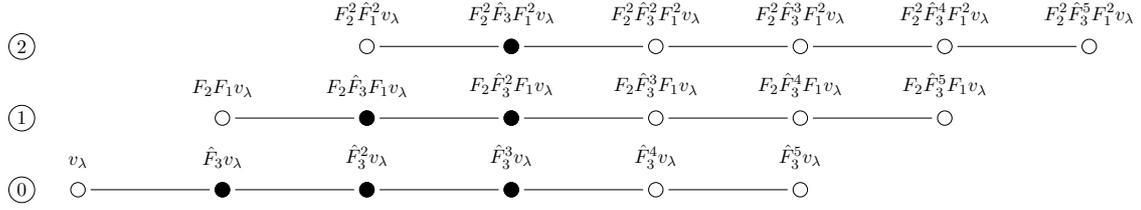
\begin{figure}[t]
\scalebox{.95}{
\begin{tikzpicture}

\draw (0,0) circle (.1cm) node[label={[label distance=.1cm,scale=.7]90:$ v_\lambda$}] {};
\draw[shorten >= 5pt, shorten <= 5pt] (0,0) to (2,0);

\draw (2,0) circle (.1cm) node[label={[label distance=.1cm,scale=.7]90:$\hat F_3  v_\lambda$}] {};
\draw[shorten >= 5pt, shorten <= 5pt] (2,0) to (4,0);

\draw (4,0) circle (.1cm) node[label={[label distance=.1cm,scale=.7]90:$\hat F_3^2  v_\lambda$}] {};
\draw[shorten >= 5pt, shorten <= 5pt] (4,0) to (6,0);

\draw (6,0) circle (.1cm) node[label={[label distance=.1cm,scale=.7]90:$\hat F_3^3  v_\lambda$}] {};
\draw[shorten >= 5pt, shorten <= 5pt] (6,0) to (8,0);

\draw (8,0) circle (.1cm) node[label={[label distance=.1cm,scale=.7]90:$\hat F_3^4  v_\lambda$}] {};
\draw[shorten >= 5pt, shorten <= 5pt] (8,0) to (10,0);

\draw (10,0) circle (.1cm) node[label={[label distance=.1cm,scale=.7]90:$\hat F_3^5 v_\lambda$}] {};

\draw (2,1) circle (.1cm) node[label={[label distance=.1cm,scale=.7]90:$F_2F_1v_\lambda$}] {};
\draw[shorten >= 5pt, shorten <= 5pt] (2,1) to (4,1);

\draw (4,1) circle (.1cm) node[label={[label distance=.1cm,scale=.7]90:$F_2\hat F_3 F_1 v_\lambda$}] {};
\draw[shorten >= 5pt, shorten <= 5pt] (4,1) to (6,1);

\draw (6,1) circle (.1cm) node[label={[label distance=.1cm,scale=.7]90:$F_2\hat F_3^2 F_1 v_\lambda$}] {};
\draw[shorten >= 5pt, shorten <= 5pt] (6,1) to (8,1);

\draw (8,1) circle (.1cm) node[label={[label distance=.1cm,scale=.7]90:$F_2\hat F_3^3 F_1 v_\lambda$}] {};
\draw[shorten >= 5pt, shorten <= 5pt] (8,1) to (10,1);

\draw (10,1) circle (.1cm) node[label={[label distance=.1cm,scale=.7]90:$F_2\hat F_3^4 F_1 v_\lambda$}] {};
\draw[shorten >= 5pt, shorten <= 5pt ] (10,1) to (12,1);

\draw (12,1) circle (.1cm) node[label={[label distance=.1cm,scale=.7]90:$F_2\hat F_3^5 F_1 v_\lambda$}] {};

\draw (4,2) circle (.1cm) node[label={[label distance=.1cm,scale=.7]90:$F_2^2\hat F^2_1 v_\lambda$}] {};
\draw[shorten >= 5pt, shorten <= 5pt] (4,2) to (6,2);

\draw (6,2) circle (.1cm) node[label={[label distance=.1cm,scale=.7]90:$F_2^2\hat F_3 F_1^2 v_\lambda$}] {};
\draw[shorten >= 5pt, shorten <= 5pt] (6,2) to (8,2);

\draw (8,2) circle (.1cm) node[label={[label distance=.1cm,scale=.7]90:$F_2^2\hat F_3^2F_1^2 v_\lambda$}] {};
\draw[shorten >= 5pt, shorten <= 5pt] (8,2) to (10,2);

\draw (10,2) circle (.1cm) node[label={[label distance=.1cm,scale=.7]90:$F_2^2\hat F_3^3 F_1^2 v_\lambda$}] {};
\draw[shorten >= 5pt, shorten <= 5pt] (10,2) to (12,2);

\draw (12,2) circle (.1cm) node[label={[label distance=.1cm,scale=.7]90:$F_2^2\hat F_3^4 F_1^2 v_\lambda$}] {};;
\draw[shorten >= 5pt, shorten <= 5pt ] (12,2) to (14,2);

\draw (14,2) circle (.1cm) node[label={[label distance=.1cm,scale=.7]90:$F_2^2\hat F_3^5 F_1^2 v_\lambda$}] {};;

\node[left,scale=.7] at (-0.5,0) {\circled{0}};
\node[left,scale=.7] at (-0.5,1) {\circled{1}};
\node[left,scale=.7] at (-0.5,2) {\circled{2}};
 \draw[fill=black] (6,2) circle (.1cm);
 \draw[fill=black] (6,1) circle (.1cm);
 \draw[fill=black] (6,0) circle (.1cm);
 \draw[fill=black] (4,1) circle (.1cm);
 \draw[fill=black] (4,0) circle (.1cm);
 \draw[fill=black] (2,0) circle (.1cm);
\end{tikzpicture}}
\caption{Structure of the basis $(u^{0}_n)_n$ of $\ker(B_1)|_{U_{0}}$ for the representation $\lambda = 2\varpi_1 + 5\varpi_2$ as in Figure \ref{uqsl3_fig_weightdiagram}. The black circles indicate the terms that contribute to the expression of the element $u^{0}_1=F_2^2\hat F_3F_1^{2} \, v_\lambda+\cdots$.}
\label{fig:structure_ui}
\end{figure}
\end{center}

\subsection{The action of $C_1$}
\label{subsec:diagC_1}
In Remark \ref{rmk_Serre_imply_ker} we observed that the kernel of $B_1$ is stable under the action of $C_1$. Furthermore for each $i=0,\ldots,\lambda_1$, $U_i$ is stable under $C_1$. The goal of this subsection is to compute the action of $C_1$ in the basis of $\ker \pi_\lambda(B_1)$ given in Proposition \ref{prop:kernel_B1}.

\begin{lem}\label{lem:action_elm_B1}
In the basis $\mathscr{B}$ of Theorem \ref{thm:basisVla} we have
\begin{align*}
F_1F_2 F_2^k\hat F_3^l F_1^{k+i}\, v_\lambda &= a_{k+1}(l,k+i) \,F_2^{k+1}\hat F_3^l F^{k+i+1} \, v_\lambda + b_{k+1}(l,k+i) \, F_2^{k}\hat F_3^{l+1} F_1^{k+i}\, v_\lambda,\\
E_2F_2  F_2^k\hat F_3^l F_1^{k+i}\, v_\lambda &= \eta_{k+1}(l,k+i) \, F_2^k\hat F_3^l\, F_1^{k+i} \, v_\lambda, \displaybreak[0]\\
F_1E_1 F_2^k\hat F_3^l F_1^{k+i}\, v_\lambda &= \alpha_k(l,k+i)\, a_k(l,k+i-1) \,  F_2^k\hat F_3^l F_1^{k+i} \, v_\lambda \\
&\qquad + \alpha_k(l,k+i)\, b_k(l,k+i-1) \, F_2^{k-1}\hat F_3^{l+1} F_1^{k+i-1} \, v_\lambda  \displaybreak[0]\\
& \qquad + \beta_k(l,k+i)\, a_{k+1}(l-1,k+i) F_2^{k+1}\hat F_3^{l-1} F_1^{k+i+1} \, v_\lambda \\
& \qquad + \beta_k(l,k+i)\, b_{k+1}(l-1,k+i) F_2^{k}\hat F_3^{l} F_1^{k+i} \, v_\lambda,\\
E_2E_1  F_2^k \hat F_3^{l}F_1^{k+i} &=\alpha_k(l,k+i)\, \eta_k(l,k+i-1) \,  F_2^{k-1}\hat F_3^l F_1^{k+i-1}\, v_\lambda \displaybreak[0] \\
& \qquad + \beta_k(l,k+i)\, \eta_{k+1}(l-1,k+i)\, F_2^k\hat F_3^{l-1} F_1^{k+i} \, v_\lambda, \\
K  F_2^{\lambda_1-i}\hat F_3^l F_1^{\lambda_1}\, v_\lambda &=     q^{\lambda_1-\lambda_2-3i} \, F_2^{\lambda_1-i}\hat F_3^l F_1^{\lambda_1}\, v_\lambda,\\
K^{-1}  F_2^{\lambda_1-i}\hat F_3^l F_1^{\lambda_1}\, v_\lambda &=     q^{\lambda_2-\lambda_1+3i} \, F_2^{\lambda_1-i}\hat F_3^l F_1^{\lambda_1}\, v_\lambda.
\end{align*}
\end{lem}
\begin{proof}
The lemma is a direct consequence of Proposition \ref{prop:actiongeneratorsinbasis}.
\end{proof}
Since $K$ acts as a multiple of the identity on each $U_i$, it suffices to determine the action of $B_1B_2$ on $U_i$.
\begin{lem}
\label{lem:action_B1B2_KerB1}
For $i\in {0,\ldots, \lambda_1}$, in  the basis $(u^i_n)_n$ of $\ker(B_1)$, we have
$$B_1B_2 \, u^i_n = A(n)u^i_{n+1}+B(n) u^i_{n} + C(n) u^i_{n-1},\qquad n=0,\ldots, \lambda_2,$$
where
\begin{align*}
A(n)&=\frac{q^{\lambda_2+i-n}(1-q^2)(1-q^{2\lambda_1+2\lambda_2-2n})}{(1-q^{2\lambda_2+2\lambda_1-2n+2})(1-q^{2\lambda_2+2i-2n})}, \\
B(n)&= -c_1 \frac{q^{2n+i-\lambda_1-\lambda_2} \, (1-q^{2\lambda_2-2n+2i})}{(1-q^2)} + \frac{c_2\, q^{\lambda_1-\lambda_2+2n-i+1}(1-q^{-2n-2i})}{(1-q^2)},\\
C(n)&= \frac{c_1c_2 \, q^{3n-3\lambda_2-i-2}(1-q^{2n})(1-q^{2\lambda_2-2n+2})(1-q^{2\lambda_1+2\lambda_2-2n+4})(1-q^{2\lambda_2+2i+2-2n})}{(1-q^2)^3(1-q^{2\lambda_2+2\lambda_1+2-2n})}.
\end{align*}
\end{lem}
\begin{proof}
Since $U_i$ is stable under $B_1B_2$ and $(u_n^i)_n$ is a basis of $U_i$, we have
$$B_1B_2 \, u_n^i = \sum_{j=0}^{\lambda_2} \nu_{j} \, u_{j}^i,$$
for certain coefficients $\nu_j$. Since $\mathscr{B}_i$ is an orthogonal basis and $u_n^i$ has a single contribution from the vectors in the upper layer of $\mathscr{B}_i$, see Remark \ref{rmk:form_ui}, we obtain that
\begin{align*}
\langle B_1B_2\, u_n^i , \,  F_2^{\lambda_1-i}\, \hat F_3^{s} \, F_1^{\lambda_1}\, v_\lambda \rangle &= \sum_{j=0}^{\lambda_2} \nu_j \langle u^i_j,F_2^{\lambda_1-i}\, \hat F_3^{s} \, F_1^{\lambda_1}\, v_\lambda  \rangle =\nu_s H_{\lambda_1-i,s,\lambda_1}^2.
\end{align*}
On the other hand, from \eqref{uqsl3_eqn_B_generators} we have
\begin{multline}
\label{eq:action_B1B2_generators}
B_1B_2  F_2^k\hat F_3^l F_1^{k+i} \, v_\lambda= F_1F_2  F_2^k\hat F_3^l F_1^{k+i} \, v_\lambda - c_1 q^{l+k+2i-1-\lambda_1} E_2F_2  F_2^k\hat F_3^l F_1^{k+i} \, v_\lambda \\
- c_2 q^{k+l-i-\lambda_2} F_1E_1  F_2^k\hat F_3^l F_1^{k+i} \, v_\lambda + c_1c_2\, q^{2l+2k+i-\lambda_1-\lambda_2-2}E_2E_1 F_2^k\hat F_3^l F_1^{k+i} \, v_\lambda.
\end{multline}
Applying Lemma \ref{lem:action_elm_B1} to \eqref{eq:action_B1B2_generators}, we verify that the action of $B_1B_2$ on the vector of the $k$-th layer $ F_2^k\hat F_3^l F_1^{k+i} \, v_\lambda$ has contributions from the $(k-1)$-th, $k$-th and $(k+1)$-th layer. Hence, Remark \ref{rmk:form_ui} implies
\begin{equation}
\label{eq:B1B2_innerp_F2F3F1}
\begin{split}
\langle B_1B_2\, u_n^i , \,  F_2^{\lambda_1-i}\, \hat F_3^{s} \, F_1^{\lambda_1}\, v_\lambda \rangle
&= \langle B_1B_2\, F_2^{\lambda_1-i}\hat F_3^nF_1^{\lambda_1}\, v_\lambda , \,  F_2^{\lambda_1-i}\, \hat F_3^{s} \, F_1^{\lambda_1}\, v_\lambda \rangle \\
& \qquad +\gamma^n_{\lambda_1-i-1,n}   \,  \langle B_1B_2\, F_2^{\lambda_1-i-1} \hat F_3^n F_1^{\lambda_1-1}\, v_\lambda, \,  F_2^{\lambda_1-i}\, \hat F_3^{s} \, F_1^{\lambda_1}\, v_\lambda \rangle \\
&\qquad+\gamma^n_{\lambda_1-i-1,n+1}   \,  \langle B_1B_2\,  F_2^{\lambda_1-i-1} \hat F_3^{n-1} F_1^{\lambda_1-1}\, v_\lambda , \,  F_2^{\lambda_1-i}\, \hat F_3^{s} \, F_1^{\lambda_1}\, v_\lambda \rangle.
\end{split}
\end{equation}
From Lemma \ref{lem:action_elm_B1} we obtain that \eqref{eq:B1B2_innerp_F2F3F1} is zero unless $s=n-1, n, n+1$. Moreover, we have
\begin{align*}
& \langle B_1B_2\, u_n^i , \,  F_2^{\lambda_1-i}\, \hat F_3^{n+1} \, F_1^{\lambda_1}\, v_\lambda \rangle   = [b_{\lambda_1-i+1}(n,\lambda_1)+\gamma^n_{\lambda_1-i-1,n+1} \, a_{\lambda_1-i}(n+1,\lambda_1-1)] H_{\lambda_1-i,n+1,\lambda_1}^2,  \displaybreak[0] \\
& \langle B_1B_2\, u_n^i , \,  F_2^{\lambda_1-i}\, \hat F_3^{n} \, F_1^{\lambda_1}\, v_\lambda \rangle  = [-c_1\, q^{n+i-1} \, \eta_{\lambda_1-i+1}(l,\lambda_1) \\
& \qquad - c_2\, q^{\lambda_1+n-2i-\lambda_2} \, \alpha_{\lambda_1-i}(n,\lambda_1) \, a_{\lambda_1-i}(n,\lambda_1-1) \\
& \qquad \qquad -c_2 \, q^{\lambda_1-2i+n-\lambda_2} \, \beta_{\lambda_1-i}(n,\lambda_1) \, b_{\lambda_1-i+1}(n-1,\lambda_1) + \gamma^n_{\lambda_1-i-1,n} \, a_{\lambda_1-i}(n,\lambda_1-1)\\
& \qquad \qquad \qquad - c_2\, q^{\lambda_1+n-2i-\lambda_2} \gamma^n_{\lambda_1-i-1,n+1} \, \beta_{\lambda_1-i-1}(n+1,\lambda_1-1)\, a_{\lambda_1-i}(n,\lambda_1-1)]H_{\lambda_1-i,n,\lambda_1}^2, \\
&\langle B_1B_2\, u_n^i , \,  F_2^{\lambda_1-i}\, \hat F_3^{n-1} \, F_1^{\lambda_1}\, v_\lambda \rangle = [c_1c_2 \, q^{\lambda_1-\lambda_2+2n-i-2} \, \beta_{\lambda_1-i}(n,\lambda_1) \, \eta_{\lambda_1-i+1}(n-1,\lambda_1) \\
&\qquad - c_2 q^{\lambda_1-\lambda_2 + n -2i-1} \, \gamma^n_{\lambda_1-i-1,n} \, \beta_{\lambda_1-i-1}(n,\lambda_1-1) \, a_{\lambda_1-i}(n-1,\lambda_1-1) ] H_{\lambda_1-i,n-1,\lambda_1}^2.
\end{align*}
Now the lemma follows from Proposition \ref{prop:actiongeneratorsinbasis} and \eqref{eq:expression_gammas}.
\end{proof}

\begin{lem}
\label{lem:action_C1_on_uin}
For $i\in {0,\ldots, \lambda_1}$, in  the basis $(u^i_n)_n$ of $\ker(B_1)$, we have
$$C_1 \, u_n^i = A(n)u^i_{n+1}+\left(B(n)+D \right) u^i_{n} + C(n) u^i_{n-1},\quad D=- c_2\frac{q^{\lambda_1-\lambda_2-3i}}{q - q^{-1}} + c_1\frac{q^{\lambda_2-\lambda_1+3i}(q + q^{-1})}{q - q^{-1}}.$$
\end{lem}
\begin{proof}
Lemma \ref{lem:action_B1B2_KerB1}, \eqref{uqsl3_eqn_C1C2} and $K$ acting as a multiple of the identity give the result.
\end{proof}

We are now ready to  find the eigenvectors of $C_1$ restricted to $\ker(B_1)|_{U_i}$. We will describe these eigenvectors as a linear combination of the vectors $u_n^i$ with explicit coefficients given in terms of dual $q$-Krawtchouk polynomials. For $N\in \mathbb{N}$ and $n=0,1,\ldots,N$, the dual $q$-Krawtchouk polynomials are given explicitly by
$$K_n(\lambda(x);c,N|q)=\frac{(q^{x-N};q)_n}{(q^{-N};q)_nq^{nx}} \, \pfq{2}{1}{q^{-n},q^{-x}}{q^{N-x-n+1}}{q}{cq^{x+1}},$$
where $\lambda(x)=q^{-x}+cq^{x-N}$, see \cite[(3.17.1)]{KoekS}. We follow the standard notation of \cite{GR} for basic hypergeometric series. The polynomials
\begin{equation}
\label{eq:rl}
r_l(\lambda(x))=(q^{-2N};q)_l \, K_l(\lambda(x);c,N|q^2),
\end{equation}
satisfy the three term recurrence relation
\begin{equation}
\label{eq:three_term_kraw}
x \, r_l(x) = r_{l+1}(x) +(1+c)q^{2l-2N}\, r_{l}(x) + c\, q^{-2N}(1-q^{2l})(1-q^{2l-2N-2}) \, r_{l-1}(x).
\end{equation}

\begin{prop}
\label{prop:basis_psi}
For $i=0,\ldots,\lambda_1$, the set $\{\psi^i_x\}_{x=0}^{\lambda_2}$ where
$$\psi^i_x=\sum_{l=0}^{\lambda_2}
\frac{c_1^l \, q^{-l(\lambda_1+2)+l(l-1)/2} \, (q^{-2\lambda_2},q^{-2\lambda_2-2\lambda_1};q^2)_l}{(q^{-2\lambda_2-2\lambda_1-2},q^{-2\lambda_2-2i};q^2)_l} \, K_{l}(\lambda(x),-c_1^{-1}c_2q^{2\lambda_1-2i+1},\lambda_2,q^2)\, u^i_l,$$
is a basis of eigenvectors of $C_1$ restricted to $\ker(B_1)|_{U_i}$. The eigenvalue of $\psi^i_x$ is
$$\eta_1 = \frac{c_1 \, \kappa^{-1} \, q (1 + q^{-2n-2}) - c_2 \, \kappa \, q^{2n}}{q - q^{-1}},$$
for $\kappa=q^{\lambda_1-3i-\lambda_2}$ and $n=x+i$.
\end{prop}

\begin{rmk}
As we pointed out in Remark \ref{rmk:ui_not_orthogonal}, the basis $(u^i_n)_n$ is not orthogonal. Still the operator $C_1$ acts tridiagonally. Moreover, if $\mathcal{B}$ is $\ast$-invariant then the basis $\{\psi^i_x\}_{x=0}^{\lambda_2}$ in Proposition \ref{prop:basis_psi} is orthogonal although, because of the non-orthogonality of $(u^i_n)_n$, this does not follow directly from the orthogonality of the dual $q$-Krawtchouk polynomials.
\end{rmk}

\begin{proof}
Assume there exist polynomials $p_n(x)$ such that $v=\sum_{l=0}^{\lambda_2} p_l(x) \, u^i_l$ is an eigenvector of $C_1$ with eigenvalue $\eta_1$, i.e. $C_1\, v = \eta_1 \, v$. From Lemma \ref{lem:action_C1_on_uin} we have
$$C_1\, v = \sum_{l=0}^{\lambda_2} \, p_l(x) ( A(l)u^i_{l+1}+\left(B(l)+D \right) u^i_{l} + C(l) u^i_{l-1}) =\sum_{l=0}^{\lambda_2} \, \eta_1 \, p_l(x)\, u^i_l.$$
Since $(u^i_l)_l$ is a basis of $\ker(B_1)|_{U_i}$ the vectors $u_l^i$ are linearly independent and hence the polynomials $p_l$ satisfy the following three term recurrence relation
$$\eta_1\, p_l(x)=C(l+1)p_{l+1}(x)+ \left(B(l)+D \right) p_l(x) + A(l-1) p_{l-1}(x).$$
If $k_l$ is the leading coefficient of $p_l$, then $P_l=k_l^{-1}p_l$ is a sequence of monic polynomials satisfying the recurrence relation
\begin{equation}
\label{eq:recurrence_P_l}
\eta_1 \, P_l(x)=P_{l+1}(x)+ \left(B(l)+D \right) P_l(x) + C(l)A(l-1)P_{l-1}(x),
\end{equation}
where
\begin{align*}
B(l)+D &= -\frac{c_1 \, q^{2l+i-\lambda_1-\lambda_2}(1-c_1^{-1}c_2\, q^{2\lambda_1-2i+1})}{(1-q^2)} - \frac{c_1\, q^{3i-\lambda_1+\lambda_2+2}}{(1-q^2)},\\
C(l)A(l-1)&=-\frac{c_1c_2\,q\,(1-q^{2l})(1-q^{2l-2\lambda_2-2})}{(1-q^2)^2},
\end{align*}
using Lemma \ref{lem:action_B1B2_KerB1} and Lemma \ref{lem:action_C1_on_uin}. We will identify the polynomials $P_l$ with the dual $q$-Krawtchouk polynomials. If we let
$$c=-c_1^{-1}c_2\, q^{2\lambda_1-2i+1}, \quad N=\lambda_2,$$
 the recurrence relation \eqref{eq:three_term_kraw} is given by
\begin{multline*}
x\, r_l(x) = r_{l+1}(x) +(1+c_1^{-1}c_2q^{2\lambda_1-2i+1})q^{2l-2\lambda_2}\, r_{l}(x)  \\
+ c_1^{-1}c_2q^{2\lambda_1-2\lambda_2-2i+1}(1-q^{2l})(1-q^{2l-2\lambda_2-2}) \, r_{l-1}(x).
\end{multline*}
If we let $\tilde r_l(x)=a^{-l} \, r_l(ax)$ with $a=-c^{-1}_1\, q^{\lambda_1-\lambda_2-i}(1-q^2)$, by a straightforward computation we obtain
\begin{multline}
\label{eq:rec_Q_a}
\left(x - \frac{c_1\, q^{3i-\lambda_1+\lambda_2+2}}{(1-q^2)}\right) \, \tilde r_l(x) = \tilde r_{l+1}(x) + (B(l)+D)\,  \tilde r_l(x) + C(l)A(l-1)\, \tilde r_{l-1}(x).
\end{multline}
If we evaluate \eqref{eq:rec_Q_a} in $\lambda(x)a^{-1}$, the eigenvalue is given by
\begin{align*}
\frac{\lambda(x)}{a}-\frac{c_1\, q^{3i-\lambda_1+\lambda_2+2}}{(1-q^2)} =\frac{c_1\, q^{3i-\lambda_1+\lambda_2+2}(1+q^{-2x-2i-2}) + c_2\, q^{\lambda_1-\lambda_2-i+2x}}{q-q^{-1}}.
\end{align*}
Therefore the polynomials $P_l(x)=\tilde{r}(\lambda(x)a^{-1})=a^{-l} \, r_l(\lambda(x))$ satisfy the recurrence \eqref{eq:recurrence_P_l} with
eigenvalue
\begin{align*}
\eta_1&=\frac{c_1 \, \kappa^{-1} \, q (1 + q^{-2n-2}) - c_2 \, \kappa \, q^{2n}}{q - q^{-1}},
\end{align*}
with  $\kappa=q^{\lambda_1-3i-\lambda_2}$ and $n=x+i$, for $x=0,\ldots,\lambda_2$. Finally, $p_l(x)=k_l\,a^{-1}r_l(\lambda(x))$. The explicit expression of $p_l$ follows from \eqref{eq:rl} and Lemma \ref{lem:action_B1B2_KerB1}.
\end{proof}

\begin{proof}[Proof of Theorem \ref{thm:branching}]
From Proposition \ref{prop:basis_psi} we obtain vectors $\psi_x^i$ for $i=0,\ldots, \lambda_1$, $x=0,\ldots,\lambda_2$ such that
$$\pi_\lambda(B_1)\, \psi_x^i = 0,\quad \text{and} \quad C_1\, \psi_x^i = \frac{c_1 \, \kappa^{-1} \, q (1 + q^{-2n-2}) - c_2 \, \kappa \, q^{2n}}{q - q^{-1}} \, \psi_x^i = \eta_1 \, \psi_x^i.$$
where $\kappa=q^{\lambda_1-3i-\lambda_2}$ and $n=x+i$, so that $\psi_x^i$ is a highest weight vector. It follows from Corollary \ref{cor:caracterizations_irreps_B} that the highest weight vector $\psi_x^i$ defines an irreducible representation of $\mathcal{B}$ of dimension $x+i+1$
$$W_{q^{\lambda_1-\lambda_2-3i},x+i}=\langle \{\, \psi_x^i, \, \pi_\lambda(B_2)\, \psi_x^i, \,  \pi_\lambda(B_2)^2\, \psi_x^i, \,\ldots \,,\,  \pi_\lambda(B_2)^{x+i}\, \psi_x^i \, \} \rangle.$$
Let $W=\oplus_{(\kappa,n)} W_{(\kappa,n)}$ where the sum is taken over $(\kappa,n)=(q^{\lambda_1-3i-\lambda_2},x+i)$ for $i=0,\ldots, \lambda_1$, $x=0,\ldots,\lambda_2$. We have that $W\subset V_\lambda$ and
\begin{equation*}
\dim W = \sum_{i,x} \dim W_{q^{\lambda_1-\lambda_2-3i},x+i}  = \frac12 (\la_1+1)(\la_2+1)(\la_1+\la_2+2) = \dim V_\lambda.
\end{equation*}
Therefore $W=V_\lambda$ and this completes the proof of the theorem.
\end{proof}

\subsection*{Acknowledgement.}
We thank Stefan Kolb for helpful discussions on this paper. Noud Aldenhoven also thanks him for his hospitality during his visit to Newcastle.

The research of  Noud Aldenhoven is supported by the Netherlands Organization
for Scientific Research (NWO) under project number 613.001.005 and by the Belgian Interuniversity
Attraction Pole Dygest P07/18.

The research of Pablo Rom\'an is supported by the Radboud Excellence Fellowship. P. Rom\'an was partially supported by CONICET grant PIP 112-200801-01533 and by SeCyT-UNC.

\begin{appendix}

\section{Proof of Theorem \ref{thm:basisVla} }
\label{apen:proof_basis}

\begin{lem}\label{lem:relinUqa}
The following relations hold in $\Uq$:
\begin{enumerate}[(i)]
\setlength\itemsep{0.3em}
\item\label{lem:relinUq-ii} $F_2\hat{F}_3[a] = \hat{F}_3[a+1]F_2,$
\item\label{lem:relinUq-iii} $E_1\hat{F}_3[a] = \hat{F}_3[a+1]E_1 + F_2\frac{(q^{a+1}K_1K_2-q^{-a-1}(K_1K_2)^{-1})}{(q-q^{-1})},$
\item\label{lem:relinUq-vi} $F_2F_3 = q F_3F_2,$
\item\label{lem:relinUq-vii} $\bigl( \hat{F}_3[a]\bigr)^\ast  = q \hat{E}_3[a] (K_1K_2)^{-1}$,
$F_3^\ast = qE_3 (K_1K_2)^{-1},$
\item\label{lem:relinUq-x} $\hat{F}_3 = F_3 \frac{qK_2-q^{-1}K_2^{-1}}{q-q^{-1}} + qF_2F_1K_2,$
\item\label{lem:relinUq-xi} $E_1{F}_3 = F_3E_1+F_2K_1,$
\end{enumerate}
\end{lem}

\begin{proof}
Straightforward verifications using \eqref{eq:uqsl3-relations} and  \eqref{eq:uqsl3-star}.
\end{proof}

\begin{cor}\label{cor:lem:relinUq-iii} For $l\in \N$ and $a\in \R$ we have
\[E_1\bigl( \hat{F}_3[a]\bigr)^l =
\bigl(\hat{F}_3[a+1]\bigr)^lE_1 +
\frac{q^l-q^{-l}}{q-q^{-1}}
F_2\bigl( \hat{F}_3[a]\bigr)^{l-1} \frac{(q^{a+2-l}K_1K_2-q^{-a-2+l}(K_1K_2)^{-1})}{(q-q^{-1})}
\]
\end{cor}

\begin{proof} By induction on $l$ using Lemma \ref{lem:relinUqa}\eqref{lem:relinUq-iii}
and \eqref{lem:relinUq-ii}.
\end{proof}

\begin{proof}[Proof of  Theorem \ref{thm:basisVla}]

By the PBW-theorem, $F_2^k \hat{F}_3^l F_1^m\, v_\la$ for $k,l,m\in \N$ spans
$V_\la$. By Proposition \ref{prop:actiongeneratorsinbasis}
\begin{equation}\label{eq:actionKibasis}
\begin{split}
K_1  F_2^k \hat{F}_3^l F_1^m\, v_\la &=
q^{\la_1+k-l-2m} F_2^k \hat{F}_3^l F_1^m\, v_\la, \\
K_2  F_2^k \hat{F}_3^l F_1^m\, v_\la &=
q^{\la_2-2k-l+m} F_2^k \hat{F}_3^l F_1^m\, v_\la.
\end{split}
\end{equation}
Since $K_i$, $i=1,2$, are self-adjoint, we find that $\langle F_2^k \hat{F}_3^l F_1^m\, v_\la, F_2^{k'} \hat{F}_3^{l'} F_1^{m'}\, v_\la \rangle = 0$ in case $k-l-2m\not= k'-l'-2m'$ or $-2k-l+m\not= -2k'-l'+m'$.  For $k'>k$ we find
\begin{multline}\label{eq:pfthmbasisVla-1}
\langle F_2^k \hat{F}_3^l F_1^m\, v_\la, F_2^{k'} \hat{F}_3^{l'} F_1^{m'}\, v_\la \rangle =
\langle (E_2K_2^{-1})^{k'}F_2^k \hat{F}_3^l F_1^m\, v_\la,  \hat{F}_3^{l'} F_1^{m'}\, v_\la \rangle \\
= q^{k'(k'+1)} \langle K_2^{-k'} E_2^{k'}F_2^k \hat{F}_3^l F_1^m\, v_\la,  \hat{F}_3^{l'} F_1^{m'}\, v_\la \rangle  = 0,
\end{multline}
since  $E_i^{k'}F_i^k \in \Uq\, E_i^{k'-k}$ for $k,k'\in \N$, $k'>k$, using also
Lemma \ref{lem:relinUq}\eqref{lem:relinUq-iv} for $a=0$,
\eqref{eq:uqsl3-relations} and
\eqref{eq:defrelhwVla}.
Because of the symmetry between $k$ and $k'$, we see that the inner product \eqref{eq:pfthmbasisVla-1} is $0$ for $k\not= k'$.
With the above remark, we find
\begin{equation*}
\langle F_2^k \hat{F}_3^l F_1^m\, v_\la, F_2^{k'} \hat{F}_3^{l'} F_1^{m'}\, v_\la \rangle = 0
\end{equation*}
in case $k\not= k'$ or  $l\not=l'$ or $m\not= m'$.

So it suffices to calculate the norm of the vectors, and see that this is non-zero
precisely for the range mentioned.
First, using the case $k=k'$ of the first part of \eqref{eq:pfthmbasisVla-1}
and that $K_2$ acts on $E_2^{k}F_2^k \hat{F}_3^l F_1^m\, v_\la$
by the scalar $q^{\la_2-l+m}$, we find
\begin{gather*}
\langle F_2^k \hat{F}_3^l F_1^m\, v_\la, F_2^{k} \hat{F}_3^{l} F_1^{m}\, v_\la \rangle =
q^{k(k+1)-k(\la_2-l+m)} \langle E_2^{k}F_2^k \hat{F}_3^l F_1^m\, v_\la,  \hat{F}_3^{l} F_1^{m}\, v_\la \rangle.
\end{gather*}
Now use Lemma \ref{lem:EkFkmodUqEk}\eqref{lem:EkFkmodUqEk-ii} for $i=2$ and next
the commutation relations of Lemma \ref{lem:relinUq}\eqref{lem:relinUq-iv}
and \eqref{eq:uqsl3-relations} to see that the $\Uq E_2$-part of
Lemma \ref{lem:EkFkmodUqEk}\eqref{lem:EkFkmodUqEk-ii} gives zero contribution.
Because of the action of $K_2$ being diagonal, we find
\begin{gather*}
\langle F_2^k \hat{F}_3^l F_1^m\, v_\la, F_2^{k} \hat{F}_3^{l} F_1^{m}\, v_\la \rangle =
\frac{(q^2;q^2)_k}{(1-q^2)^{2k}} (q^{-2(\la_2-l+m)};q^2)_k
(-1)^k q^{3k}
\langle  \hat{F}_3^l F_1^m\, v_\la,  \hat{F}_3^{l} F_1^{m}\, v_\la \rangle
\end{gather*}
Next we write
\begin{gather*}
\langle  \hat{F}_3^l F_1^m\, v_\la,  \hat{F}_3^{l} F_1^{m}\, v_\la \rangle =
\langle  \hat{F}_3^l F_1^m\, v_\la,  F_1^{m} \hat{F}_3^{l} \, v_\la \rangle =
q^{m(m+1)} q^{-m(\la_1-l)}
\langle E_1^{m}  \hat{F}_3^l F_1^m\, v_\la,  \hat{F}_3^{l} \, v_\la \rangle
\end{gather*}
using Lemma \ref{lem:relinUq}\eqref{lem:relinUq-i}, the $\ast$-structure \eqref{eq:uqsl3-star},
\eqref{eq:uqsl3-relations} and \eqref{eq:actionKibasis}.
Following Mudrov \cite[\S 8]{Mudr} we replace $\hat{F}_3^l$ on the left by $F_3^l$.
First use Lemma \ref{lem:relinUqa}\eqref{lem:relinUq-x}
\begin{gather*}
\langle E_1^{m}  \hat{F}_3^l F_1^m\, v_\la,  \hat{F}_3^{l} \, v_\la \rangle =
\frac{q^{2+\la_2-l+m}-q^{-2-\la_2+l-m}}{q-q^{-1}}
\langle E_1^{m} F_3 \hat{F}_3^{l-1} F_1^m\, v_\la,  \hat{F}_3^{l} \, v_\la \rangle \\
+ q^{2+\la_2-l+m}
\langle E_1^{m} F_2F_1 \hat{F}_3^{l-1} F_1^m\, v_\la,  \hat{F}_3^{l} \, v_\la \rangle
\end{gather*}
In the second term, move $F_2$ to the left using \eqref{eq:uqsl3-relations}, and
then the other side so that is essentially an $E_2$ which we can move through,
by Lemma \ref{lem:relinUq}\eqref{lem:relinUq-iv}, to the highest weight vector, and hence gives
zero.
This we can repeat, since $F_2$ also $q$-commutes with $F_3$ by
Lemma \ref{lem:relinUqa}\eqref{lem:relinUq-vi}. This yields
\begin{gather*}
\langle E_1^{m}  \hat{F}_3^l F_1^m\, v_\la,  \hat{F}_3^{l} \, v_\la \rangle =
\frac{(-1)^l q^{l(2+\la_2+m)l} q^{-\frac12 l(l-1)}}{(1-q^2)^l}
(q^{-\la_2-2-m};q^2)_l
\langle E_1^{m} F_3^{l} F_1^m\, v_\la,  \hat{F}_3^{l} \, v_\la \rangle.
\end{gather*}
Using Lemma \ref{lem:relinUqa}\eqref{lem:relinUq-xi}, and moving $F_2$ to the other
side, where $F_2^\ast$ kills $\hat{F}_3^{l} \, v_\la$, we see
\begin{gather*}
\langle E_1^{m} F_3^{l} F_1^m\, v_\la,  \hat{F}_3^{l} \, v_\la \rangle =
(-1)^mq^{-m(m-2)+m\la_1} \frac{(q^2;q^2)_m}{(1-q^2)^{2m}} (q^{-2\la_1};q^2)_m
\langle F_3^{l} \, v_\la,  \hat{F}_3^{l} \, v_\la \rangle
\end{gather*}
by Lemma \ref{lem:EkFkmodUqEk}\eqref{lem:EkFkmodUqEk-ii}.
Assume $l\geq 1$, so it remains to calculate
\begin{multline*}
\langle  {F}_3^l \, v_\la,  \hat{F}_3^{l} \, v_\la \rangle =
\langle {F}_3^{l-1} \, v_\la,  ({F}_3)^\ast \hat{F}_3^{l} \, v_\la \rangle
 =  q^{1-(\la_1+\la_2-2l)} \langle
{F}_3^{l-1} \, v_\la,  \left( E_2E_1 - E_1E_2
\right) \hat{F}_3^{l} \, v_\la \rangle  \\
=q^{1-(\la_1+\la_2-2l)} \langle
{F}_3^{l-1} \, v_\la,  E_2E_1
 \hat{F}_3^{l} \, v_\la \rangle
\end{multline*}
where we use Lemma \ref{lem:relinUqa}\eqref{lem:relinUq-vii}, the diagonal action
of $K_i$ and the fact that the action of $E_1E_2$ is zero by
Lemma \ref{lem:relinUq}\eqref{lem:relinUq-iv} and \eqref{eq:defrelhwVla}.
By Corollary \ref{cor:lem:relinUq-iii} for $a=0$ and \eqref{eq:defrelhwVla} we find
\begin{gather*}
E_1 \hat{F}_3^l \, v_\la = \frac{q^l-q^{-l}}{q-q^{-1}}
\frac{q^{2+\la_1+\la_2-l}-q^{-2-\la_1-\la_2+l}}{q-q^{-1}} F_2 \hat{F}_3^{l-1} \, v_\la
\end{gather*}
and next applying $E_2$, using \eqref{eq:uqsl3-relations}, \eqref{eq:defrelhwVla}
and Lemma \ref{lem:relinUq}\eqref{lem:relinUq-iv} we find
\begin{gather*}
E_2E_1 \hat{F}_3^l \, v_\la = \frac{q^l-q^{-l}}{q-q^{-1}}
\frac{q^{2+\la_1+\la_2-l}-q^{-2-\la_1-\la_2+l}}{q-q^{-1}}
\frac{q^{\la_2-l+1}-q^{-\la_2+l-1}}{q-q^{-1}}
\hat{F}_3^{l-1} \, v_\la,
\end{gather*}
so that
\begin{multline*}
\langle  {F}_3^l \, v_\la,  \hat{F}_3^{l} \, v_\la \rangle =
q^{1-(\la_1+\la_2-2l)} \frac{q^l-q^{-l}}{q-q^{-1}} \\
\times \frac{q^{2+\la_1+\la_2-l}-q^{-2-\la_1-\la_2+l}}{q-q^{-1}}
\frac{q^{\la_2-l+1}-q^{-\la_2+l-1}}{q-q^{-1}}
\langle  {F}_3^{l-1} \, v_\la,  \hat{F}_3^{l-1} \, v_\la \rangle.
\end{multline*}
Iterating, since we normalize $\langle v_\la,v_\la\rangle =1$,  we
find
\begin{gather*}
\langle {F}_3^l \, v_\la,  \hat{F}_3^{l} \, v_\la \rangle =
q^{l(\la_2+7)} q^{-\frac12 l(l+1)}
\frac{(q^2;q^2)_l}{(1-q^2)^{3l}} (q^{-2\la_2};q^2)_l
(q^{-2(\la_1+\la_2+1)};q^2)_l.
\end{gather*}
Note that this expression is positive for $0\leq l \leq \la_2$
and equals zero for $l>\la_2$. Collecting all the intermediate results gives the explicit expression for
the norm of the basis elements.
\end{proof}

\end{appendix}


\begin{thebibliography}{99}

\bibitem{AKR15}
N.~Aldenhoven, E.~Koelink, P.~Rom\'an, \emph{Matrix valued orthogonal polynomials for the quantum analogue of $(\SU(2)\times\SU(2), \diag)$}, preprint, arXiv:1507.03426.

\bibitem{DN98}
M.S.~Dijkhuizen, M.~Noumi, \emph{A family of quantum projective spaces and related $q$-hypergeometric orthogonal polynomials}, Trans. Amer. Math. Soc., {\bf 350} (1998), 3269--3296.

\bibitem{DS99}
M.S.~Dijkhuizen, J.V.~Stokman, \emph{Some limit transitions between BC type orthogonal polynomials interpreted on quantum complex Grassmannians}, Publ. Res. Inst. Math. Sci., {\bf 35} (1999), 451--500.

\bibitem{GR}
G.~Gasper, M.~Rahman.  \emph{Basic Hypergeometric Series}, volume~96,  Cambridge University Press, 2nd edition, 2004.

\bibitem{HvP13}
G.~Heckman, M.~van Pruijssen, \emph{Matrix valued orthogonal polynomials for Gelfand pairs of rank one}, Tohoku Math. J. (2), to appear, \texttt{arXiv:1310.5134}.

\bibitem{KS97}
A.~Klimyk, K.~Schm{\"u}dgen, \emph{Quantum Groups and Their Representations}, Texts and Monographs in Physics, Springer-Verlag, Berlin, 1997.

\bibitem{KoekS} R.~Koekoek, R.F.~Swarttouw, {The Askey-scheme
of hypergeometric orthogonal polynomials and its $q$-analogue}, online at
\texttt{http://aw.twi.tudelft.nl/\~{}koekoek/askey.html}, Report
98-17, Technical University Delft, 1998.

\bibitem{KvPR12}
E.~Koelink, M.~van Pruijssen, P.~Rom\'an, \emph{Matrix-valued orthogonal polynomials related to $(\SU(2) \times \SU(2), \diag)$}. Int. Math. Res. Not. IMRN, {\bf 24} (2012), 5673--5730.

\bibitem{KvPR13}
E.~Koelink, M.~van Pruijssen, R.~Rom\'an, \emph{Matrix-valued orthogonal polynomials related to $(\SU(2) \times \SU(2), \diag)$}, II, Publ. Res. Inst. Math. Sci., {\bf 49} (2013), 271--312.

\bibitem{Kol14}
S.~Kolb, \emph{Quantum symmetric Kac-Moody pairs}, Adv. Math., {\bf 267} (2014), 395--469.

\bibitem{Koo93}
T.~Koornwinder, \emph{Askey-Wilson polynomials as zonal spherical functions on the $\SU(2)$ quantum group}, SIAM J. Math. Anal., {\bf 24} (1993), 795--813.

\bibitem{KL08}
S.~Kolb, G.~Letzter, \emph{The center of quantum symmetric pair coideal subalgebras}, Rep. Theory, {\bf 12} (2008), 294--326.

\bibitem{Let99}
G.~Letzter, \emph{Symmetric pairs for quantized enveloping algebras}, J. Alg., {\bf 220} (1999), 729--767.

\bibitem{Let00}
G.~Letzter, \emph{Harish-{C}handra modules for quantum symmetric pairs}, Rep. Theory, {\bf 4} (2000), 64--96.

\bibitem{Let02}
G.~Letzter, \emph{Coideal subalgebras and quantum symmetric pairs}, New Directions in Hopf Algebras, Cambridge University Press, Cambridge, {\bf 43} (2002), 117--166.

\bibitem{Let03}
G.~Letzter, \emph{Quantum symmetric pairs and their zonal spherical functions}, Transform. Groups, {\bf 8} (2003),261--292.

\bibitem{Let04}
G.~Letzter, \emph{Quantum zonal spherical functions and Macdonald polynomials}, Adv. Math., {\bf 189} (2004), 88--147.

\bibitem{Let08}
G.~Letzter, \emph{Invariant differential operators for quantum symmetric spaces}, Mem. Amer. Math. Soc., {\bf 193} (2008).

\bibitem{Mudr}
A.~Mudrov,
\emph{Orthogonal basis for the Shapovalov form on {$U_q(\mathfrak{sl}(n + 1))$}},
Rev. Math. Phys. \textbf{27} (2015), no. 2, 1550004, 23 pp.

\bibitem{NDS97}
M.~Noumi, M.S.~Dijkhuizen, T.~Sugitani, \emph{Multivariable Askey-Wilson polynomials and quantum complex Grassmannians}, Fields Inst. Commun., {\bf 14} (1997), 167--177.

\bibitem{NS95}
M.~Noumi, T.~Sugitani, \emph{Quantum symmetric spaces and related $q$-orthogonal polynomials}, Group Theoretical Methods in Physics, World Science Publishing, River Edge, (1994), 28--40.

\bibitem{Nou96}
M.~Noumi, \emph{Macdonald's symmetric polynomials as zonal spherical functions on some quantum homogeneous spaces}, Adv. Math., {\bf 123} (1996), 16--77.

\bibitem{OS}
A. A. Oblomkov, J. V. Stokman, \emph{Vector valued spherical functions and Macdonald-Koornwinder polynomials}, Compos. Math. {\bf 141} (2005), no. 5, 1310--1350.

\bibitem{Sug99}
T. Sugitani, \emph{Zonal spherical functions on quantum Grassmann manifolds}, J. Math. Sci Univ. Tokyo, {\bf 6} (1999), 335--369.

\bibitem{vPrui12}
M.~van Pruijssen, \emph{Matrix valued orthogonal polynomials related to compact Gel'fand paris of rank one}, PhD Thesis, Radboud Universiteit, (2012).
\end{thebibliography}
\end{document}